\documentclass[11pt]{article}
\usepackage{amsmath,amssymb,amsthm}
\usepackage{enumitem}
\usepackage[affil-it]{authblk}
\usepackage{a4wide}
\usepackage{color}
\usepackage[dvipsnames]{xcolor}
\usepackage{hyperref}
\usepackage{xfrac}
\usepackage{tikz}
\usepackage{tikz}\usetikzlibrary{arrows.meta}
\usepackage{caption} 
\usetikzlibrary{arrows}
\usepackage{authblk}
\usepackage{graphicx}
\usepackage{mathrsfs}
\usepackage{mathtools}
\usepackage{rotating}

\usepackage{multirow}%

\usepackage[title]{appendix}%
\usepackage{textcomp}%
\usepackage{manyfoot}%
\usepackage{booktabs}%
\usepackage{algorithm}%
\usepackage{algorithmicx}%
\usepackage{algpseudocode}%
\usepackage{listings}%
\usepackage{mathtools}

\newtheorem{thm}[equation]{Theorem}
\newtheorem{cl}[equation]{Claim}

\newtheorem{prop}[equation]{Proposition}
\newtheorem{lemma}[equation]{Lemma}
\theoremstyle{definition}
\newtheorem{defn}[equation]{Definition}
\newtheorem{remark}[equation]{Remark}

\numberwithin{equation}{section}

\newcommand\degg{\, \mathsf{deg} \,}

\title{$C^*$- Colored graph algebras}

\author[1]{Farrokh Razavinia \thanks{This research was in part also supported by a grant from IPM (No. 1404140052)}
}
\affil[1]{School of Mathematics, Institute for Research in Fundamental Sciences (IPM), P.O. Box: 19395‐5746, Tehran, Iran}
\affil[]{f.razavinia@phystech.edu}
\date{}

\begin{document}

\maketitle

\begin{abstract}
Following our previous works on $C^*$-graph algebras and the associated Cuntz-Krieger graph families, in this paper we will try to have a look at the colored version of these structures and to see what a $C^*$-colored graph algebra might mean by employing some constructive examples very close to the toy example used in our previous works, and we also will try to study their graph theoretical properties as possible.
\newline\newline
\textbf{MSC Numbers (2020)}: Primary 46L05; Secondary 46L89,58B34.
 \hfill \newline
\end{abstract}

\section{Notations}
Throughout this paper, our ground field will be the field of complex numbers $\mathbb C$, and we will use $*$ to present the complex conjugate transpose of vectors and matrices. 

We will use $\overset{\leftarrow}{\degg}_{i}$ in order to show the entry degree of the vertex $v_i$, and by entry degree we mean the number of edges that are entering to the vertex $v_i$, and analogously we have exit degree of the vertex $v_i$, which will be considered by $\overset{\rightarrow}{\degg}_{i}$.

Other required notations will be introduced within the original text of the paper.

\section{Introduction}
In \cite{R243} we introduced and studied a non-entangled $(4i-6)$-qubit quantum system based on the Cuntz-Krieger graph families related to the directed graphs $\mathcal{G}_i$ associated with the coordinate ring of the $i\times i$ quantum matrix algebra $M_i(q)$, for $i\in\{2,3,\cdots\}$, and the multiplier Hopf $*$-graph algebras $\mathcal{O}(GL_n(\mathbb C),\Delta)$ and $\mathcal{O}(SL_n(\mathbb C),\Delta)$, for $\Delta$ defined as follows based on the commuting matrices $\pi_i$ with the adjacency matrix $\Pi_i$ of the directed graphs $\mathcal{G}_i$.
\begin{align}
        \Delta:&\mathcal{O}(G))\to M(\mathcal{O}(G)\otimes \mathcal{O}(G))\\&\hspace*{0.3cm} E_{i,j}\longmapsto E_{k,h}\otimes E_{o,r}:=E_{\ell,m},
    \end{align}

for $G$ playing the role of $GL(n)$ and $SL(n)$, and $\ell=P_{o}^{k}$ and $m= P_{r}^{h}$, linearly expandable on whole of $M_n(\mathbb C)$, where $GL(n)$ and $SL(n)$ are defined.
\begin{remark}
    Note that while working with $\pi_n$, actually our space is a $n^2\times n^2$ block matrix, consisting of $n^n$, $n\times n$ matrices, $n$ blocks in each row and column.  In what comes, our base will be the set $\{1,\cdots,n^2\}$ divided into $n$ partitions $P_1$ to $P_n$, such that each $P_i$, $i\in\{1,\cdots,n\}$, also is divided into $n$ partitions $P_{i}^{j}$ for $j\in\{1,\cdots,n\}$.
\end{remark}

In \cite{R243} our initial example was the graph $\mathcal{G}_2:=\mathcal{G}(\Pi_2)$ associated with the coordinate ring of $M_q(2)$, with the set of vertices and edges as $\mathcal{G}_{2}^{0}=\{x_{11}:=u, x_{12}:=v, x_{22}:=k, x_{21}:=w\}$ and $\mathcal{G}_{2}^{1}=\{x_{11}\overrightarrow{\sim}x_{12}:=e, x_{11}\overrightarrow{\sim}x_{21}:=f, x_{12}\overrightarrow{\sim}x_{22}:=h, x_{21}\overrightarrow{\sim}x_{22}:=g, x_{12}\overrightarrow{\sim}x_{21}:=i, x_{21}\overrightarrow{\sim}x_{12}:=j\}$, 

\vspace*{0.1cm}

\hspace*{2.7cm} \begin{tikzpicture}\label{Gra:2}
%[
%box/.style={draw,rectangle,minimum size=2cm,text width=1.5cm,align=left}]
\tikzset{vertex/.style = {shape=circle,draw,minimum size=0.7em}}
\tikzset{edge/.style = {->,> = latex'}}
% vertices
\node[vertex] (a) at  (0,0) {$u$};
\node[vertex] (b) at  (4,3) {$v$};
\node[vertex] (c) at  (8,0) {$k$};
\node[vertex] (d) at  (4,-3) {$w$};
%\node[vertex] (a1) at (1.5,0) {};
%\node[vertex] (a2) at (3,0) {};
%edges
\draw[edge] (a) to node[above] {e} (b);
\draw[edge] (b) to node[above] {h} (c);
\draw[edge] (a) to node[below] {f} (d);
\draw[edge] (d) to node[below] {g} (c);

%	\draw[edge] (a)  to[bend left] (a1);
%	\draw[edge] (a1) to[bend left] (a);

%	\draw[edge] (a1) to[bend left] (a2);
%	\draw[edge] (a2) to[bend left] (a1);

%	\path (a2) to node {\dots} (c);
%	\node [shape=circle,minimum size=1.5em] (a3) at (4.5,0) {};
%	\draw[edge] (a2) to[bend left] (a3);
%	\draw[edge] (a3) to[bend left] (a2);

%	\node [shape=circle,minimum size=0.7em] (c1) at (6.5,0) {};
\draw[edge] (b) to[bend left] node[right] {i} (d);
\draw[edge] (d) to[bend left] node[left] {j} (b);
%\node [below=0.9cm, align=flush center,text width=8cm] at (d)
%{
%	$M_1$
%}

\end{tikzpicture}
\captionof{figure}{\textbf{Directed locally connected graph related to $\Pi_2$}}

\vspace*{0.1cm}

Here our initial example will be almost the same, but the complete one and undirected. Consider the following undirected graph on 4 vertices.

\vspace*{0.1cm}

\hspace*{3.3cm} \begin{tikzpicture}\label{Gra:2}
%[
%box/.style={draw,rectangle,minimum size=2cm,text width=1.5cm,align=left}]
\tikzset{vertex/.style = {shape=circle,draw,minimum size=0.7em}}
\tikzset{edge/.style = {-,> = latex'}}
% vertices
\textcolor{red}{\node[vertex] (a) at  (0,0) {$u$};}
\textcolor{blue}{\node[vertex] (b) at  (4,3) {$v$};}
\textcolor{green}{\node[vertex] (c) at  (8,0) {$k$};}
\textcolor{blue}{\node[vertex] (d) at  (4,-3) {$w$};}
%\node[vertex] (a1) at (1.5,0) {};
%\node[vertex] (a2) at (3,0) {};
%edges
\textcolor{teal}{\draw[edge] (a) to node[above] {e} (b);}
\textcolor{blue}{\draw[edge] (b) to node[above] {h} (c);}
\textcolor{blue}{\draw[edge] (a) to node[below] {f} (d);}
\textcolor{teal}{\draw[edge] (d) to node[below] {g} (c);}

%	\draw[edge] (a)  to[bend left] (a1);
%	\draw[edge] (a1) to[bend left] (a);

%	\draw[edge] (a1) to[bend left] (a2);
%	\draw[edge] (a2) to[bend left] (a1);

%	\path (a2) to node {\dots} (c);
%	\node [shape=circle,minimum size=1.5em] (a3) at (4.5,0) {};
%	\draw[edge] (a2) to[bend left] (a3);
%	\draw[edge] (a3) to[bend left] (a2);

%	\node [shape=circle,minimum size=0.7em] (c1) at (6.5,0) {};
%\textcolor{red}{\draw[edge] (b) to node[below] {\hspace{-0.3cm}\begin{turn}{0}$i$\end{turn}} (d);}
%\draw[edge] (b) to node[below] {i} (d);
%\draw[edge] (a) to node[above] {j} (c);
\textcolor{black}{\draw[edge] (a) to node[below] {\hspace{0cm}\begin{turn}{0}$j$\end{turn}} (c);}
%\draw[edge] (b) to[bend left] node[right] {i} (d);
%\draw[edge] (d) to[bend left] node[left] {j} (b);
%\node [below=0.9cm, align=flush center,text width=8cm] at (d)
%{
%	$M_1$
%}

\end{tikzpicture}
\captionof{figure}{\textbf{Two connected graph $\mathsf{Sq}_2$}}\label{Fig::1}

\vspace*{0.1cm}

Note that the 2 connected graph presented in Figure \ref{Fig::1}, with $K(\mathsf{Sq}_2)=2$ and the vertex and edge chromatic number $\chi_{v,e}(\mathsf{Sq}_2)=3$, is not a special graph on its own, but could become a special graph by enlarging it in a predefined way as follows.

\vspace*{0.5cm}

\hspace*{3.1cm} \begin{tikzpicture}\label{Gra:2}
%[
%box/.style={draw,rectangle,minimum size=2cm,text width=1.5cm,align=left}]
\tikzset{vertex/.style = {shape=circle,draw,minimum size=0.7em}}
\tikzset{edge/.style = {-,> = latex'}}
% vertices
\textcolor{red}{\node[vertex] (a) at  (0,0) {$u$};}
\textcolor{blue}{\node[vertex] (b) at  (4,3) {$v$};}
\textcolor{green}{\node[vertex] (c) at  (8,0) {$k$};}
\textcolor{blue}{\node[vertex] (d) at  (4,-3) {$w$};}
\textcolor{red}{\node[vertex] (e) at  (0,3) {$l$};}
\textcolor{red}{\node[vertex] (f) at  (0,-3) {$m$};}
\textcolor{green}{\node[vertex] (g) at  (8,3) {$n$};}
\textcolor{green}{\node[vertex] (h) at  (8,-3) {$o$};}
%\node[vertex] (a1) at (1.5,0) {};
%\node[vertex] (a2) at (3,0) {};
%edges
\textcolor{red}{\draw[edge] (a) to node[above] {e} (b);}
%\draw[edge] (b) to node[above] {h} (c);
\textcolor{blue}{\draw[edge] (b) to node[below] {\hspace{1.0cm}\begin{turn}{-360}$h$\end{turn}} (c);}
%\draw[edge] (a) to node[below] {f} (d);
\textcolor{blue}{\draw[edge] (a) to node[below] {\hspace{1.6cm}\begin{turn}{20}$f$\end{turn}} (d);}
\textcolor{teal}{\draw[edge] (d) to node[below] {g} (c);}
%\draw[edge] (g) to node[below] {n} (f);
\textcolor{teal}{\draw[edge] (g) to node[above] {\hspace{0.2cm}\begin{turn}{20}$n$\end{turn}} (f);}
\textcolor{red}{\draw[edge] (g) to node[below] {\hspace{0.2cm}\begin{turn}{20}$m$\end{turn}} (d);}
\textcolor{black}{\draw[edge] (e) to node[above] {\hspace{-1.2cm}\begin{turn}{20}$o$\end{turn}} (d);}
\textcolor{red}{\draw[edge] (e) to node[below] {\hspace{0.2cm}\begin{turn}{20}$p$\end{turn}} (h);}
\textcolor{black}{\draw[edge] (f) to node[below] {\hspace{1.2cm}\begin{turn}{320}$r$\end{turn}} (b);}
\textcolor{teal}{\draw[edge] (h) to node[below] {\hspace{1.2cm}\begin{turn}{20}$p$\end{turn}} (b);}
%	\draw[edge] (a)  to[bend left] (a1);
%	\draw[edge] (a1) to[bend left] (a);

%	\draw[edge] (a1) to[bend left] (a2);
%	\draw[edge] (a2) to[bend left] (a1);

%	\path (a2) to node {\dots} (c);
%	\node [shape=circle,minimum size=1.5em] (a3) at (4.5,0) {};
%	\draw[edge] (a2) to[bend left] (a3);
%	\draw[edge] (a3) to[bend left] (a2);

%	\node [shape=circle,minimum size=0.7em] (c1) at (6.5,0) {};
%\textcolor{red}{\draw[edge] (b) to node[below] {\hspace{-0.3cm}\begin{turn}{0}$i$\end{turn}} (d);}
%\draw[edge] (b) to node[below] {i} (d);
%\draw[edge] (a) to node[above] {j} (c);
\draw[edge] (a) to node[above] {\hspace{3cm}\begin{turn}{0}$j$\end{turn}} (c);
%\draw[edge] (b) to[bend left] node[right] {i} (d);
%\draw[edge] (d) to[bend left] node[left] {j} (b);
%\node [below=0.9cm, align=flush center,text width=8cm] at (d)
%{
%	$M_1$
%}

\end{tikzpicture}

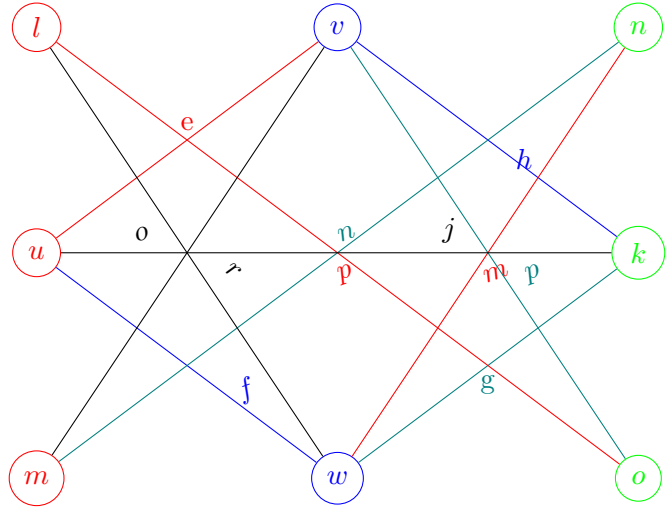
\captionof{figure}{\textbf{Four connected graph $\mathsf{Sq}_3$}}\label{Fig::2}

\vspace*{0.5cm}

Note that in order to obtain the graph presented in Figure \ref{Fig::2}, we added four more vertices, so that the graph obtained by connecting some certain vertices to each other still possesses the same vertex chromatic number $\chi_v(\mathsf{Sq_3})=3$, but having $K(\mathsf{Sq}_3)=4$, and the edge chromatic number $\chi_e(\mathsf{Sq}_3)=4$.

\vspace*{0.5cm}

\hspace*{1.7cm} \begin{tikzpicture}\label{Gra:2}
%[
%box/.style={draw,rectangle,minimum size=2cm,text width=1.5cm,align=left}]
\tikzset{vertex/.style = {shape=circle,draw,minimum size=0.7em}}
\tikzset{edge/.style = {-,> = latex'}}
% vertices
\textcolor{red}{\node[vertex] (a) at  (0,0) {$u$};}
\textcolor{blue}{\node[vertex] (b) at  (4,3) {$v$};}
\textcolor{green}{\node[vertex] (c) at  (11,0) {$k$};}
\textcolor{blue}{\node[vertex] (d) at  (4,-6) {$w$};}
\textcolor{red}{\node[vertex] (e) at  (0,3) {$l$};}
\textcolor{red}{\node[vertex] (f) at  (0,-6) {$m$};}
\textcolor{green}{\node[vertex] (g) at  (11,3) {$n$};}
\textcolor{green}{\node[vertex] (h) at  (11,-6) {$o$};}
%\node[vertex] (a1) at (1.5,0) {};
%\node[vertex] (a2) at (3,0) {};
\textcolor{red}{\node[vertex] (a_1) at  (0,-3) {$u_1$};}
\textcolor{blue}{\node[vertex] (b_1) at  (7,3) {$v_1$};}
\textcolor{green}{\node[vertex] (c_1) at  (11,-3) {$k_1$};}
\textcolor{blue}{\node[vertex] (d_1) at  (7,-6) {$w_1$};}
%edges

\textcolor{brown}{\draw[edge] (a) to node[above] {e} (b);}
%\draw[edge] (b) to node[above] {h} (c);
\textcolor{blue}{\draw[edge] (b) to node[below] {\hspace{1.0cm}\begin{turn}{-360}$h$\end{turn}} (c);}
%\draw[edge] (a) to node[below] {f} (d);
\textcolor{blue}{\draw[edge] (a) to node[below] {\hspace{1.6cm}\begin{turn}{20}$f$\end{turn}} (d);}
\textcolor{teal}{\draw[edge] (d) to node[below] {g} (c);}
%\draw[edge] (g) to node[below] {n} (f);
\textcolor{teal}{\draw[edge] (g) to node[above] {\hspace{0.2cm}\begin{turn}{20}$n$\end{turn}} (f);}
\textcolor{red}{\draw[edge] (g) to node[below] {\hspace{0.2cm}\begin{turn}{20}$m$\end{turn}} (d);}
\textcolor{black}{\draw[edge] (e) to node[above] {\hspace{-1.2cm}\begin{turn}{20}$o$\end{turn}} (d);}
\textcolor{red}{\draw[edge] (e) to node[below] {\hspace{0.2cm}\begin{turn}{20}$p$\end{turn}} (h);}
\textcolor{black}{\draw[edge] (f) to node[below] {\hspace{1.2cm}\begin{turn}{320}$r$\end{turn}} (b);}
\textcolor{teal}{\draw[edge] (h) to node[below] {\hspace{1.2cm}\begin{turn}{20}$p$\end{turn}} (b);}
%	\draw[edge] (a)  to[bend left] (a1);
%	\draw[edge] (a1) to[bend left] (a);

%	\draw[edge] (a1) to[bend left] (a2);
%	\draw[edge] (a2) to[bend left] (a1);

%	\path (a2) to node {\dots} (c);
%	\node [shape=circle,minimum size=1.5em] (a3) at (4.5,0) {};
%	\draw[edge] (a2) to[bend left] (a3);
%	\draw[edge] (a3) to[bend left] (a2);
\textbf{\textcolor{red}{\draw[edge] (b_1) to node[below] {\hspace{1.9cm}\begin{turn}{20}$p_1$\end{turn}} (c);}}
\textbf{\textcolor{red}{\draw[edge] (b_1) to node[below] {\hspace{1.9cm}\begin{turn}{20}$p_1$\end{turn}} (c);}}
\textbf{\textcolor{red}{\draw[edge] (b_1) to node[below] {\hspace{1.9cm}\begin{turn}{20}$p_1$\end{turn}} (c);}}
\textbf{\textcolor{red}{\draw[edge] (b_1) to node[below] {\hspace{1.9cm}\begin{turn}{20}$p_1$\end{turn}} (c);}}
%	\node [shape=circle,minimum size=0.7em] (c1) at (6.5,0) {};
\textbf{\textcolor{yellow}{\draw[edge] (c) to node[below] {\hspace{0.6cm}\begin{turn}{20}$g_1$\end{turn}} (d_1);}}
\textbf{\textcolor{yellow}{\draw[edge] (c) to node[below] {\hspace{0.6cm}\begin{turn}{20}$g_1$\end{turn}} (d_1);}}
\textbf{\textcolor{yellow}{\draw[edge] (c) to node[below] {\hspace{0.6cm}\begin{turn}{20}$g_1$\end{turn}} (d_1);}}
\textbf{\textcolor{yellow}{\draw[edge] (c) to node[below] {\hspace{0.6cm}\begin{turn}{20}$g_1$\end{turn}} (d_1);}}
%\textcolor{red}{\draw[edge] (b) to node[below] {\hspace{-0.3cm}\begin{turn}{0}$i$\end{turn}} (d);}
\textbf{\textcolor{black}{\draw[edge] (g) to node[below] {\hspace{0.6cm}\begin{turn}{20}$d_1$\end{turn}} (d_1);}}
\textbf{\textcolor{black}{\draw[edge] (g) to node[below] {\hspace{0.6cm}\begin{turn}{20}$d_1$\end{turn}} (d_1);}}
\textbf{\textcolor{black}{\draw[edge] (g) to node[below] {\hspace{0.6cm}\begin{turn}{20}$d_1$\end{turn}} (d_1);}}
\textbf{\textcolor{black}{\draw[edge] (g) to node[below] {\hspace{0.6cm}\begin{turn}{20}$d_1$\end{turn}} (d_1);}}
%\textcolor{red}{\draw[edge] (b) to node[below] {\hspace{-0.3cm}\begin{turn}{0}$i$\end{turn}} (d);}
\textbf{\textcolor{black}{\draw[edge] (g) to node[below] {\hspace{0.6cm}\begin{turn}{20}$d_1$\end{turn}} (d_1);}}
\textbf{\textcolor{black}{\draw[edge] (g) to node[below] {\hspace{0.6cm}\begin{turn}{20}$d_1$\end{turn}} (d_1);}}
\textbf{\textcolor{black}{\draw[edge] (g) to node[below] {\hspace{0.6cm}\begin{turn}{20}$d_1$\end{turn}} (d_1);}}
\textbf{\textcolor{black}{\draw[edge] (g) to node[below] {\hspace{0.6cm}\begin{turn}{20}$d_1$\end{turn}} (d_1);}}
%\textcolor{red}{\draw[edge] (b) to node[below] {\hspace{-0.3cm}\begin{turn}{0}$i$\end{turn}} (d);}
\textbf{\textcolor{blue}{\draw[edge] (c_1) to node[below] {\hspace{1.6cm}\begin{turn}{20}$b_1$\end{turn}} (b_1);}}
\textbf{\textcolor{blue}{\draw[edge] (c_1) to node[below] {\hspace{1.6cm}\begin{turn}{20}$b_1$\end{turn}} (b_1);}}
\textbf{\textcolor{blue}{\draw[edge] (c_1) to node[below] {\hspace{1.6cm}\begin{turn}{20}$b_1$\end{turn}} (b_1);}}
\textbf{\textcolor{blue}{\draw[edge] (c_1) to node[below] {\hspace{1.6cm}\begin{turn}{20}$b_1$\end{turn}} (b_1);}}
%\textcolor{red}{\draw[edge] (b) to node[below] {\hspace{-0.3cm}\begin{turn}{0}$i$\end{turn}} (d);}
\textbf{\textcolor{red}{\draw[edge] (c_1) to node[below] {\hspace{1.6cm}\begin{turn}{20}$c_1$\end{turn}} (b);}}
\textbf{\textcolor{red}{\draw[edge] (c_1) to node[below] {\hspace{1.6cm}\begin{turn}{20}$c_1$\end{turn}} (b);}}
\textbf{\textcolor{red}{\draw[edge] (c_1) to node[below] {\hspace{1.6cm}\begin{turn}{20}$c_1$\end{turn}} (b);}}
\textbf{\textcolor{red}{\draw[edge] (c_1) to node[below] {\hspace{1.6cm}\begin{turn}{20}$c_1$\end{turn}} (b);}}
%\textcolor{red}{\draw[edge] (b) to node[below] {\hspace{-0.3cm}\begin{turn}{0}$i$\end{turn}} (d);}
\textbf{\textcolor{yellow}{\draw[edge] (c_1) to node[above] {\hspace{3.2cm}\begin{turn}{20}$a_1$\end{turn}} (d);}}
\textbf{\textcolor{yellow}{\draw[edge] (c_1) to node[above] {\hspace{3.2cm}\begin{turn}{20}$a_1$\end{turn}} (d);}}
\textbf{\textcolor{yellow}{\draw[edge] (c_1) to node[above] {\hspace{3.2cm}\begin{turn}{20}$a_1$\end{turn}} (d);}}
\textbf{\textcolor{yellow}{\draw[edge] (c_1) to node[above] {\hspace{3.2cm}\begin{turn}{20}$a_1$\end{turn}} (d);}}
%\textcolor{red}{\draw[edge] (b) to node[below] {\hspace{-0.3cm}\begin{turn}{0}$i$\end{turn}} (d);}
\textbf{\textcolor{black}{\draw[edge] (c_1) to node[above] {\hspace{0.2cm}\begin{turn}{20}$e_1$\end{turn}} (a_1);}}
\textbf{\textcolor{black}{\draw[edge] (c_1) to node[above] {\hspace{0.2cm}\begin{turn}{20}$e_1$\end{turn}} (a_1);}}
\textbf{\textcolor{black}{\draw[edge] (c_1) to node[above] {\hspace{0.2cm}\begin{turn}{20}$e_1$\end{turn}} (a_1);}}
\textbf{\textcolor{black}{\draw[edge] (c_1) to node[above] {\hspace{0.2cm}\begin{turn}{20}$e_1$\end{turn}} (a_1);}}
%\textcolor{red}{\draw[edge] (b) to node[below] {\hspace{-0.3cm}\begin{turn}{0}$i$\end{turn}} (d);}
\textbf{\textcolor{teal}{\draw[edge] (c_1) to node[below] {\hspace{0.2cm}\begin{turn}{20}$f_1$\end{turn}} (d_1);}}
\textbf{\textcolor{teal}{\draw[edge] (c_1) to node[below] {\hspace{0.2cm}\begin{turn}{20}$f_1$\end{turn}} (d_1);}}
\textbf{\textcolor{teal}{\draw[edge] (c_1) to node[below] {\hspace{0.2cm}\begin{turn}{20}$f_1$\end{turn}} (d_1);}}
\textbf{\textcolor{teal}{\draw[edge] (c_1) to node[below] {\hspace{0.2cm}\begin{turn}{20}$f_1$\end{turn}} (d_1);}}
%\textcolor{red}{\draw[edge] (b) to node[below] {\hspace{-0.3cm}\begin{turn}{0}$i$\end{turn}} (d);}
\textbf{\textcolor{teal}{\draw[edge] (a_1) to node[below] {\hspace{0.2cm}\begin{turn}{20}$h_1$\end{turn}} (b_1);}}
\textbf{\textcolor{teal}{\draw[edge] (a_1) to node[below] {\hspace{0.2cm}\begin{turn}{20}$h_1$\end{turn}} (b_1);}}
\textbf{\textcolor{teal}{\draw[edge] (a_1) to node[below] {\hspace{0.2cm}\begin{turn}{20}$h_1$\end{turn}} (b_1);}}
\textbf{\textcolor{teal}{\draw[edge] (a_1) to node[below] {\hspace{0.2cm}\begin{turn}{20}$h_1$\end{turn}} (b_1);}}
%\textcolor{red}{\draw[edge] (b) to node[below] {\hspace{-0.3cm}\begin{turn}{0}$i$\end{turn}} (d);}
\textbf{\textcolor{red}{\draw[edge] (a_1) to node[above] {\hspace{0.2cm}\begin{turn}{20}$k_1$\end{turn}} (d_1);}}
\textbf{\textcolor{red}{\draw[edge] (a_1) to node[above] {\hspace{0.2cm}\begin{turn}{20}$k_1$\end{turn}} (d_1);}}
\textbf{\textcolor{red}{\draw[edge] (a_1) to node[above] {\hspace{0.2cm}\begin{turn}{20}$k_1$\end{turn}} (d_1);}}
\textbf{\textcolor{red}{\draw[edge] (a_1) to node[above] {\hspace{0.2cm}\begin{turn}{20}$k_1$\end{turn}} (d_1);}}
%\textcolor{red}{\draw[edge] (b) to node[below] {\hspace{-0.3cm}\begin{turn}{0}$i$\end{turn}} (d);}
\textbf{\textcolor{yellow}{\draw[edge] (a_1) to node[above] {\hspace{0.2cm}\begin{turn}{20}$o_1$\end{turn}} (b);}}
\textbf{\textcolor{yellow}{\draw[edge] (a_1) to node[above] {\hspace{0.2cm}\begin{turn}{20}$o_1$\end{turn}} (b);}}
\textbf{\textcolor{yellow}{\draw[edge] (a_1) to node[above] {\hspace{0.2cm}\begin{turn}{20}$o_1$\end{turn}} (b);}}
\textbf{\textcolor{yellow}{\draw[edge] (a_1) to node[above] {\hspace{0.2cm}\begin{turn}{20}$o_1$\end{turn}} (b);}}
%\textcolor{red}{\draw[edge] (b) to node[below] {\hspace{-0.3cm}\begin{turn}{0}$i$\end{turn}} (d);}
\textbf{\textcolor{brown}{\draw[edge] (a_1) to node[above] {\hspace{0.2cm}\begin{turn}{20}$o_1$\end{turn}} (d);}}
\textbf{\textcolor{brown}{\draw[edge] (a_1) to node[above] {\hspace{0.2cm}\begin{turn}{20}$o_1$\end{turn}} (d);}}
\textbf{\textcolor{brown}{\draw[edge] (a_1) to node[above] {\hspace{0.2cm}\begin{turn}{20}$o_1$\end{turn}} (d);}}
\textbf{\textcolor{brown}{\draw[edge] (a_1) to node[above] {\hspace{0.2cm}\begin{turn}{20}$o_1$\end{turn}} (d);}}
%\textcolor{red}{\draw[edge] (b) to node[below] {\hspace{-0.3cm}\begin{turn}{0}$i$\end{turn}} (d);}
\textbf{\textcolor{brown}{\draw[edge] (a) to node[above] {\hspace{0.2cm}\begin{turn}{20}$p_1$\end{turn}} (d_1);}}
\textbf{\textcolor{brown}{\draw[edge] (a) to node[above] {\hspace{0.2cm}\begin{turn}{20}$p_1$\end{turn}} (d_1);}}
\textbf{\textcolor{brown}{\draw[edge] (a) to node[above] {\hspace{0.2cm}\begin{turn}{20}$p_1$\end{turn}} (d_1);}}
\textbf{\textcolor{brown}{\draw[edge] (a) to node[above] {\hspace{0.2cm}\begin{turn}{20}$p_1$\end{turn}} (d_1);}}
%\textcolor{red}{\draw[edge] (b) to node[below] {\hspace{-0.3cm}\begin{turn}{0}$i$\end{turn}} (d);}
\textbf{\textcolor{blue}{\draw[edge] (e) to node[above] {\hspace{0.2cm}\begin{turn}{20}$z_1$\end{turn}} (d_1);}}
\textbf{\textcolor{blue}{\draw[edge] (e) to node[above] {\hspace{0.2cm}\begin{turn}{20}$z_1$\end{turn}} (d_1);}}
\textbf{\textcolor{blue}{\draw[edge] (e) to node[above] {\hspace{0.2cm}\begin{turn}{20}$z_1$\end{turn}} (d_1);}}
\textbf{\textcolor{blue}{\draw[edge] (e) to node[above] {\hspace{0.2cm}\begin{turn}{20}$z_1$\end{turn}} (d_1);}}
%\textcolor{red}{\draw[edge] (b) to node[below] {\hspace{-0.3cm}\begin{turn}{0}$i$\end{turn}} (d);}
\textbf{\textcolor{yellow}{\draw[edge] (a) to node[above] {\hspace{1.2cm}\begin{turn}{20}$w_1$\end{turn}} (b_1);}}
\textbf{\textcolor{yellow}{\draw[edge] (a) to node[above] {\hspace{1.2cm}\begin{turn}{20}$w_1$\end{turn}} (b_1);}}
\textbf{\textcolor{yellow}{\draw[edge] (a) to node[above] {\hspace{1.2cm}\begin{turn}{20}$w_1$\end{turn}} (b_1);}}
\textbf{\textcolor{yellow}{\draw[edge] (a) to node[above] {\hspace{1.2cm}\begin{turn}{20}$w_1$\end{turn}} (b_1);}}
%\textcolor{red}{\draw[edge] (b) to node[below] {\hspace{-0.3cm}\begin{turn}{0}$i$\end{turn}} (d);}
\textbf{\textcolor{black}{\draw[edge] (h) to node[above] {\hspace{0.9cm}\begin{turn}{20}$q_1$\end{turn}} (b_1);}}
\textbf{\textcolor{black}{\draw[edge] (h) to node[above] {\hspace{0.9cm}\begin{turn}{20}$q_1$\end{turn}} (b_1);}}
\textbf{\textcolor{black}{\draw[edge] (h) to node[above] {\hspace{0.9cm}\begin{turn}{20}$q_1$\end{turn}} (b_1);}}
\textbf{\textcolor{black}{\draw[edge] (h) to node[above] {\hspace{0.9cm}\begin{turn}{20}$q_1$\end{turn}} (b_1);}}
%\textcolor{red}{\draw[edge] (b) to node[below] {\hspace{-0.3cm}\begin{turn}{0}$i$\end{turn}} (d);}
\textbf{\textcolor{brown}{\draw[edge] (f) to node[above] {\hspace{0.9cm}\begin{turn}{20}$r_1$\end{turn}} (b_1);}}
\textbf{\textcolor{brown}{\draw[edge] (f) to node[above] {\hspace{0.9cm}\begin{turn}{20}$r_1$\end{turn}} (b_1);}}
\textbf{\textcolor{brown}{\draw[edge] (f) to node[above] {\hspace{0.9cm}\begin{turn}{20}$r_1$\end{turn}} (b_1);}}
\textbf{\textcolor{brown}{\draw[edge] (f) to node[above] {\hspace{0.9cm}\begin{turn}{20}$r_1$\end{turn}} (b_1);}}
%\draw[edge] (b) to node[below] {i} (d);
%\draw[edge] (a) to node[above] {j} (c);
\draw[edge] (a) to node[above] {\hspace{0.2cm}\begin{turn}{0}$j$\end{turn}} (c);
%\draw[edge] (b) to[bend left] node[right] {i} (d);
%\draw[edge] (d) to[bend left] node[left] {j} (b);
%\node [below=0.9cm, align=flush center,text width=8cm] at (d)
%{
%	$M_1$
%}

\end{tikzpicture}

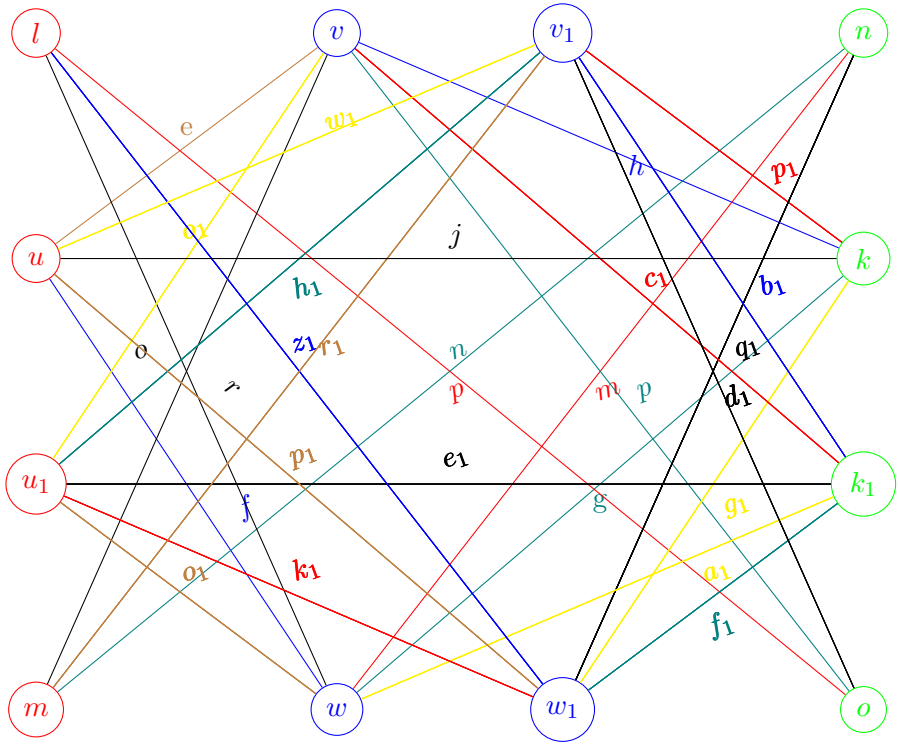
\captionof{figure}{\textbf{Six connected graph $\mathsf{Sq}_4$}}\label{Fig::3}

\vspace{0.5cm}

\vspace*{0.3cm}
Note that $\mathsf{Sq}_4$ has the same vertex chromatic number $\chi_v(\mathsf{Sq}_4)=3$, but its edge chromatic number is $\chi_e(\mathsf{Sq}_4)=6$, and it is six connected. But how did we obtain $\mathsf{Sq}_4$?

We will answer this in the next section \ref{Section:3}.

\vspace*{0.3cm}

\begin{remark}\label{Rem:1}
    There are three types of vertices in $\mathsf{Sq}_3$. That is, vertices with degrees 2, 3, and 4. Vertices with degree 3 are important to us. In \cite{RH24}, we used two operations overlay $+$ and connect $\to$ as one operator
\begin{align}
	& \Gamma_1+\Gamma_2:=\left(\Gamma_{1}^{0}\cup \Gamma_{2}^{0}, \Gamma_{1}^{1}\cup \Gamma_{2}^{1} \right) \\& \Gamma_1\to \Gamma_2:=\left(\Gamma_{1}^{0}\cup \Gamma_{2}^{0}, \Gamma_{1}^{1}\cup \Gamma_{2}^{1} \right).
	\end{align}
 for graphs $\Gamma_1=(\Gamma_{1}^{0},\Gamma_{1}^{1})$ and $\Gamma_2=(\Gamma_{2}^{0},\Gamma_{2}^{1})$. Here in this paper, we still will use these kinds of operator, but in a slightly different way!
\end{remark}
\begin{remark}\label{Rem:2}
    \begin{enumerate}[label=\roman*)]
        \item \label{Rem:2:1} Let $\Gamma_0$ be the null graph, and let $\Gamma=(\Gamma^{0},\Gamma^{1})\in\mathsf{G}_s/\{\Gamma_0\}$ be a graph with set of vertices $\Gamma^{0}$ and set of edges $\Gamma^{1}$, such that $\#\Gamma^{0}=n^2$. Then we have
        \begin{equation}\label{No.Edges}
            \#\Gamma^1=1+(n-1)(4n-3),
        \end{equation}
        for $n\in\{3,4,\cdots\}$.
        \item For $\Gamma$ as in the first part, let $\mathscr{H}_{\Gamma}$ be the set consisting of all Hamiltonian paths in $\Gamma$. Then the number of Hamiltonian paths will be as follows
        \begin{equation}\label{No.H.Paths}
            \#\mathscr{H}_{\Gamma}=10n+(2n-1)(2n-9).
        \end{equation}
    \end{enumerate}
\end{remark}
Note that later on, we will try to prove equations \ref{No.Edges} and \ref{No.H.Paths}. The proof is almost straightforward.
 \section{Graph colored algebra}\label{Section:3}
 Consider the set $\mathsf{C}_g$ consisting of a set of colors, a set of points (playing as the vertices of a graph) and a set of lines (playing as the edges of a graph). From $\mathrm{C}_g$, the colors red, green, blue and black are very important to us. Now consider the set of graphs $\mathsf{G}_c=\{\mathsf{G}_{c_i} ~|~ i\in\{3,4,\cdots\}\}$ with vertices arranged as two $(i-2)\times (i-2)$ and $2\times 2$ entangled surface lattices (here by a surface lattice we mean a lattice with points on its border) with vertices colored in four different colors red, blue, green, and black, satisfying the following rules.
 \begin{enumerate}
     \item The $2\times 2$ surface lattice will be the cover layer and the $(i-2)\times (i-2)$ surface lattice will be the inner layer of the graphs in $\mathsf{G}_c$. So, in general, the graphs in $\mathsf{G}_c$ will consist of only two layers!
     \item In $\mathsf{G}_{c_i}$ the number of source and sink vertices will be $i-2$ for each of them.
     \item The source and sink vertices will be in the blue category. That is, the blue vertices will provide us with the sink and the source vertices.
     \item Vertices are colored red, blue and green, but there is only $i$ number of each color in the graph $\mathsf{G}_{c_i}$.
     \item From these $i$ number of vertices, colored with one color, $(i-2)$ of them will be in each layer, outer and the inner layer.
     \item In each layer, the opposite color vertices other than the blue ones are two-sided connected, and the opposite color vertices from the outer layer, are in $\leftarrow$ and out $\rightarrow$ connected to the opposite color vertices of the inner layer.
 \end{enumerate}

 From these graphs, the null graph will be the one with no edges and all its vertices colored black. Let $\mathsf{G}_s$ be a subset of $\mathsf{G}_c$ consisting of only the null $n\times n$ graph, and the graphs colored blue, red and green, so that the vertices on the side edges are colored red and green, with each edge colored with just one color, and the other vertices are colored blue, as illustrated in Figures \ref{Fig::2} and \ref{Fig::3}.

 \begin{remark}
    For $n\in\{3,4,\cdots\}$, the vertices of the graphs in $\mathsf{G}_s$ possess three kinds of degrees. The highest degree is $2n-2$, and the middle is $2n-3$. However, the smallest degree is $n-1$. Consider these sets of vertices with $V_1, V_2$, and $V_3$ respectively.
 \end{remark}

 As stated in Remark \ref{Rem:1}, here in this paper we still will use the connect and overlay operators, but this time the definition will be a little closer to their original definitions, and they could be determined as follows.

 \begin{defn}\label{Def::2:}
     Let $V_b, V_g, V_r$ be the set of blue, green and red vertices, respectively, and let $\Gamma_1, \Gamma_2$ be in $\mathsf{G}_s$. Then the connect and overlay operators will be defined as follows for graphs $\Gamma_1=(\Gamma_{1}^{0},\Gamma_{1}^{1})$ and $\Gamma_2=(\Gamma_{2}^{0},\Gamma_{2}^{1})$.
     \begin{align}
	& \Gamma_1+\Gamma_2:=\left(\Gamma_{1}^{0}\cup \Gamma_{2}^{0}, \Gamma_{1}^{1}\cup \Gamma_{2}^{1} \right)\label{Over:1} \\& \Gamma_1\to \Gamma_2:=\left(\Gamma_{1}^{0}\cup \Gamma_{2}^{0}, (\Gamma_{1}^{1}\cup \Gamma_{2}^{1})/\{V_b\to V_b ~\& ~ V_r\to V_r ~\& ~ V_g\to V_g \} \right),\label{Connect:1}\\& \Gamma_0 \to \Gamma:=(\Gamma^0,\Gamma^1), \qquad \forall \Gamma=(\Gamma^0,\Gamma^1)\in\mathsf{G}_s \ \text{and}\ \Gamma_0\ \text{the null graph}. \label{Null:1}
	\end{align}
 \end{defn}
 \begin{remark}
     The relation \ref{Null:1} means that, under the overlay operator, the isolated vertices will be omitted.
 \end{remark}

 \begin{prop}
     The set $\mathsf{G}_s$ has a unital nondegenerate $*$-monoid algebra structure equipped with binary operations \ref{Over:1}, \ref{Connect:1}, and \ref{Null:1} defined in Definition \ref{Def::2:}, with the identity element $\Gamma_0$, and the diagrammatic illustration as in \cite[Figure 1]{R242}.
 \end{prop}
 \begin{proof}
     Using the relation \ref{Null:1}, it is easy to see that $\Gamma_0$ is the identity element.

     It is not too difficult to see that for graphs $\Gamma_1$ and $\Gamma_2$ in $\mathsf{G}_s$, using the operations \ref{Over:1} and \ref{Connect:1} will result another graph in $\mathsf{G}_s$. 
 \end{proof}
 \subsection{$C^*$-colored graph algebras}
In order to move to the $C^*$-graph algebra case, we need to work with the directed graphs. Hence, we need to make some changes in our graphs in $\mathsf{G}_s$. Note that the relations from the Definition \ref{Def::2:}, already provide some suggestions. So, we proceed as follows.

\begin{remark}
    In order to proceed smoothly, let us have some notation. 
    \begin{enumerate}
        \item Consider the outer layer with $\mathscr{L}_o$ and the inner layer with $\mathscr{L}_i$. Consider the red and green vertices in the outer layer with $r_{\mathscr{L}_o}$ and $g_{\mathscr{L}_o}$ respectively. On the other hand, consider the red and green vertices in the inner layer with $r_{\mathscr{L}_i}$ and $g_{\mathscr{L}_i}$ respectively.
        \item In each graph $\mathsf{G}_{c_i}$, there are $(i-2)$ vertices of each color in each layer. Consider the set consisting of these red and green vertices with $O_r:=\{O_{r_j} ~|~ j\in\{1,\cdots,i-2\}\}, O_g:=\{O_{g_j} ~|~ j\in\{1,\cdots,i-2\}\}, I_r:=\{I_{r_j} ~|~ j\in\{1,\cdots,i-2\}  \},$ and $I_g:=\{I_{g_j} ~|~ j\in\{1,\cdots,i-2\}\}$, in outer and inner layers, respectively.
    \end{enumerate}
    
\end{remark}

\begin{defn}
    The orientation of the graphs in $\mathsf{G}_s$ will be naturally defined as follows 
   % \begin{equation}\label{Ori::1}
  %      V_3\to V_3\to V_1\to V_2\to V_2
   % \end{equation}
    \begin{enumerate}
        \item The vertices of the graph $\Gamma\in\mathsf{G}_s$ will be connected unless they are in the same color category.
        \item $r_{\mathscr{L}_i} \leftrightarrow g_{\mathscr{L}_i} \ \& \ r_{\mathscr{L}_o} \leftrightarrow g_{\mathscr{L}_o},$
        \item $I_{g_{j}} \rightarrow O_{r_{j'}} \ \& \ O_{r_{j''}} \rightarrow I_{g_{j}},$ 
       \item $I_{r_{j}} \rightarrow O_{g_{j'}} \ \& \ O_{g_{j''}} \rightarrow I_{r_{j}}.$ 
    \end{enumerate}
    for $j''\neq j\neq j'$ (meaning that the above connections are between different vertices with different colors).
 %   for $V_{r2}$ and $V_{g2}$ the set of red and green vertices in second degree order, respectively, and $V_{b_{so}}$ and $V_{b_{si}}$ the set of blue source and sink vertices, respectively.
\end{defn}
Therefore, for example, for the graph $Sq_3$, Figure \ref{Fig::2} will change as follows. 

\vspace{0.5cm}

\vspace*{0.3cm}

\hspace*{2.4cm} \begin{tikzpicture}\label{Gra:2}
%[
%box/.style={draw,rectangle,minimum size=2cm,text width=1.5cm,align=left}]
\tikzset{vertex/.style = {shape=circle,draw,minimum size=0.7em}}
\tikzset{edge/.style = {->,> = latex'}}
% vertices
\textcolor{green}{\node[vertex] (a) at  (1,0) {$v_4$};}
\textcolor{blue}{\node[vertex] (b) at  (4,2) {$v_2$};}
\textcolor{red}{\node[vertex] (c) at  (7,0) {$v_3$};}
\textcolor{blue}{\node[vertex] (d) at  (4,-2) {$v_1$};}
\textcolor{red}{\node[vertex] (e) at  (0,3) {$v_8$};}
\textcolor{red}{\node[vertex] (f) at  (0,-3) {$v_5$};}
\textcolor{green}{\node[vertex] (g) at  (8,3) {$v_7$};}
\textcolor{green}{\node[vertex] (h) at  (8,-3) {$v_6$};}
%\node[vertex] (a1) at (1.5,0) {};
%\node[vertex] (a2) at (3,0) {};
%edges
\textcolor{black}{\draw[edge] (a) to node[above] {\hspace{0.2cm}\begin{turn}{20}$e_{1}$\end{turn}} (b);}
\textcolor{black}{\draw[edge] (f) to node[below] {\hspace{-1.7cm}\begin{turn}{-360}$e_2$\end{turn}} (h);}
\textcolor{black}{\draw[edge] (h) to node[below] {\hspace{1.7cm}\begin{turn}{-360}$e_3$\end{turn}} (f);}
\textcolor{blue}{\draw[edge] (e) to node[above] {\hspace{-1.7cm}\begin{turn}{-360}$e_4$\end{turn}} (g);}
\textcolor{blue}{\draw[edge] (g) to node[above] {\hspace{1.7cm}\begin{turn}{-360}$e_5$\end{turn}} (e);}
%\draw[edge] (b) to node[above] {h} (c);
\textcolor{red}{\draw[edge] (c) to node[below] {\hspace{1.0cm}\begin{turn}{-360}$e_6$\end{turn}} (b);}
\textcolor{black}{\draw[edge] (c) to node[right] {\hspace{0.0cm}\begin{turn}{-360}$e_7$\end{turn}} (g);}
\textcolor{brown}{\draw[edge] (h) to node[right] {\hspace{0.0cm}\begin{turn}{-360}$e_8$\end{turn}} (c);}
\textcolor{teal}{\draw[edge] (a) to node[left] {\hspace{1.0cm}\begin{turn}{-360}$e_9$\end{turn}} (e);}
\textcolor{red}{\draw[edge] (f) to node[left] {\hspace{1.0cm}\begin{turn}{-360}$e_{10}$\end{turn}} (a);}
%f
\textcolor{brown}{\draw[edge] (e) to node[below] {\hspace{0.0cm}\begin{turn}{-360}$e_{11}$\end{turn}} (b);}
\textcolor{yellow}{\draw[edge] (g) to node[below] {\hspace{0.0cm}\begin{turn}{-360}$e_{12}$\end{turn}} (b);}
\textcolor{yellow}{\draw[edge] (d) to node[above] {\hspace{0.0cm}\begin{turn}{-360}$e_{13}$\end{turn}} (f);}
\textcolor{blue}{\draw[edge] (d) to node[above] {\hspace{0.0cm}\begin{turn}{-360}$e_{14}$\end{turn}} (h);}
%\draw[edge] (a) to node[below] {f} (d);
\textcolor{brown}{\draw[edge] (d) to node[below] {\hspace{1.6cm}\begin{turn}{20}$e_{15}$\end{turn}} (a);}
\textcolor{teal}{\draw[edge] (d) to node[below] {\hspace{0.2cm}\begin{turn}{20}$e_{16}$\end{turn}} (c);}
%\draw[edge] (g) to node[below] {n} (f);
\textcolor{teal}{\draw[edge] (g) to node[above] {\hspace{0.2cm}\begin{turn}{20}$e_{17}$\end{turn}} (f);}
\textcolor{teal}{\draw[edge] (f) to node[below] {\hspace{-0.7cm}\begin{turn}{20}$e_{18}$\end{turn}} (g);}
\textcolor{red}{\draw[edge] (d) to node[below] {\hspace{0.2cm}\begin{turn}{20}$e_{19}$\end{turn}} (g);}
\textcolor{black}{\draw[edge] (d) to node[above] {\hspace{-1.2cm}\begin{turn}{20}$e_{20}$\end{turn}} (e);}
\textcolor{red}{\draw[edge] (e) to node[below] {\hspace{0.2cm}\begin{turn}{20}$e_{21}$\end{turn}} (h);}
\textcolor{red}{\draw[edge] (h) to node[above] {\hspace{-0.7cm}\begin{turn}{20}$e_{22}$\end{turn}} (e);}
\textcolor{blue}{\draw[edge] (f) to node[below] {\hspace{1.2cm}\begin{turn}{320}$e_{23}$\end{turn}} (b);}
\textcolor{teal}{\draw[edge] (h) to node[below] {\hspace{1.2cm}\begin{turn}{20}$e_{24}$\end{turn}} (b);}
%	\draw[edge] (a)  to[bend left] (a1);
%	\draw[edge] (a1) to[bend left] (a);

%	\draw[edge] (a1) to[bend left] (a2);
%	\draw[edge] (a2) to[bend left] (a1);

%	\path (a2) to node {\dots} (c);
%	\node [shape=circle,minimum size=1.5em] (a3) at (4.5,0) {};
%	\draw[edge] (a2) to[bend left] (a3);
%	\draw[edge] (a3) to[bend left] (a2);

%	\node [shape=circle,minimum size=0.7em] (c1) at (6.5,0) {};
%\textcolor{red}{\draw[edge] (b) to node[below] {\hspace{-0.3cm}\begin{turn}{0}$i$\end{turn}} (d);}
%\draw[edge] (b) to node[below] {i} (d);
%\draw[edge] (a) to node[above] {j} (c);
\textcolor{blue}{\draw[edge] (a) to node[above] {\hspace{3cm}\begin{turn}{0}$e_{25}$\end{turn}} (c);}
\textcolor{blue}{\draw[edge] (c) to node[above] {\hspace{-3cm}\begin{turn}{0}$e_{26}$\end{turn}} (a);}
%\draw[edge] (b) to[bend left] node[right] {i} (d);
%\draw[edge] (d) to[bend left] node[left] {j} (b);
%\node [below=0.9cm, align=flush center,text width=8cm] at (d)
%{
%	$M_1$
%}

\end{tikzpicture}

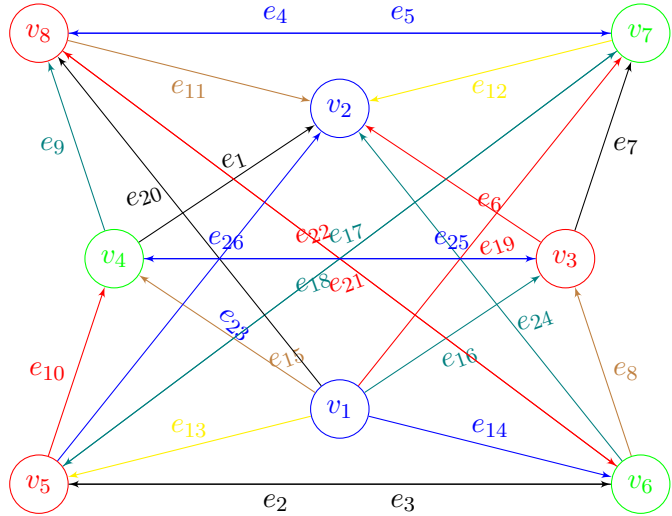
\captionof{figure}{\textbf{Five connected directed graph $\mathsf{Sq}_{3}^{d}$}}\label{Fig::4}

\vspace{1cm}

The next step is to determine the Hamiltonian paths involved in $\mathsf{Sq}_{3}^{d}$.
\begin{lemma}\label{Lem::3}
     $\mathsf{Sq}_{3}^{d}$ has 28 Hamiltonian paths.
\end{lemma}
\begin{proof}
    It is easy to observe that the Hamiltonian paths of $\mathsf{Sq}_{3}^{d}$ are as follows.

    \begin{align*}
        &w\to o\to l\to n\to m\to u\to k\to v,\quad w\to m\to n\to l\to o\to k\to u\to v,\\& w\to k\to u\to l\to o\to m\to n\to v,\quad w\to u\to k\to n\to m\to o\to l\to v.
    \end{align*}

    \begin{align*}
    & w\to n\to m\to u\to l\to o\to k\to v,\quad w\to n\to m\to o\to k\to u\to l\to v, \\& w\to k\to n\to m\to u\to l\to o\to v,\quad w\to l\to n\to m\to o\to k\to u\to v,\\&w\to o\to k\to n\to m\to u\to l\to v,\quad w\to u\to l\to n\to m\to o\to k\to v,\\&w\to l\to o\to k\to n\to m\to u\to v,\quad w\to k\to u\to l\to n\to m\to o\to v,\\& w\to u\to l\to o\to k\to n\to m\to v,\quad w\to o\to k\to u\to l\to n\to m\to v,\\&w\to m\to u\to l\to o\to k\to n\to v,\quad w\to m\to o\to k\to u\to l\to n\to v,\\&
    \end{align*}

    \begin{align*}
    &w\to u\to l\to o\to k\to n\to m\to v,\quad w\to u\to k\to n\to l\to o\to m\to v,\\&w\to m\to u\to l\to o\to k\to n\to v,\quad w\to m\to u\to k\to n\to l\to o\to v,\\&w\to n\to m\to u\to l\to o\to k\to v,\quad w\to o\to m\to u\to k\to n\to l\to v,\\&w\to k\to n\to m\to u\to l\to o\to v,\quad w\to l\to o\to m\to u\to k\to n\to v,\\&w\to o\to k\to n\to m\to u\to l\to v,\quad w\to n\to l\to o\to m\to u\to k\to v,\\&w\to l\to o\to k\to n\to m\to u\to v,\quad w\to k\to n\to l\to o\to m\to u\to v,\\&
    \end{align*}

\end{proof}
\begin{remark}
    Note that even in the above set of Hamiltonian paths, some sort of symmetry is apparent!
\end{remark}

The plan is to associate each special edge involved in each special Hamiltonian path a partial (matrix) isometry $\mathscr{S}_i$ and to each vertex an orthogonal projection $\mathscr{P}_i$, for $i\in I$ the ordered set, to see if they satisfy the relations of graph Cuntz-Krieger families, and if they produce infinite or finite graph $C^*$-algebras.

Recall that a graph $C^*$-algebra of a finite directed graph $\Gamma=(\Gamma^0,\Gamma^1)$ will be defined as follows.
\begin{defn}\label{Defn::1...}
For a finite directed graph $\Gamma$, and a finite or infinite dimensional Hilbert space $\mathcal{H}$, the set of mutually orthogonal projections $p_v\in\mathcal{H}$ for all $v\in\Gamma^0$ together with partial isometries $s_e\in\mathcal{H}$ for all $e\in\Gamma^1$ satisfying the relations

\begin{enumerate}\label{Equ:S}
    \item $s_{e}^{*}s_e=p_{r(e)}$ for all edges $e\in\Gamma^1$,\label{Equ:S:1}
    \item $p_v=\sum_{s(e)=v}s_es_{e}^{*}$ for the case when $v\in\Gamma^0$ is not a sink,\label{Equ:S:2}
\end{enumerate}
will be called a Cuntz-Krieger $\Gamma$-family in $C^*$-algebra $\mathcal{C}$.
\end{defn}
We have the following definition.
\begin{defn}\cite{BEVW22}\label{Def:GCA}
    For a finite directed graph $\Gamma=(\Gamma^0,\Gamma^1)$, the graph $C^*$-algebra $\mathcal{C}^*(\Gamma)$ is the universal $C^*$-algebra generated by a Cuntz-Krieger $\Gamma$-family $\{P_v,S_e\}$.
\end{defn}
Note that in our case, we will mostly focus on the $n\times n$ matrices with the entries in $\{0,1\}$. By considering the relations from the definition \ref{Defn::1...} of a Cuntz-Krieger $\Gamma$-family in $\mathcal{H}$, of partial isometries $S_i$ and orthogonal projections $P_i$, for $i\in\{1,\cdots,n\}$, with mutually orthogonal ranges, we will have the following Cuntz-Krieger relations
\begin{equation}\label{Equ:KC}
s_{i}^{*}s_i=\sum_{j=1}^{n}a_{ij}s_js_{j}^{*}.
\end{equation}

Note that we will only consider the nondegenerate $*$-algebras, and hence $C^*$-algebras. We have the following definition.
\begin{defn}\cite{R242}\label{Def::..5}
    For the $n\times n$ matrix $\Pi\in M_n(0,1)$, the Cuntz-Krieger algebra $\mathcal{K}_{\Pi}$ will be defined as the (nondegenerate) $C^*$-algebra generated by a universal Cuntz-Krieger $\Gamma$-family $S_i$ for $i\in\{1,\cdots,n\}$ that is satisfied in \ref{Equ:KC}.
\end{defn}

Following definition \ref{Def::..5}, note that the adjacency matrix $\mathcal{A}_{3}^{d}$ of $\mathsf{Sq}_{3}^{d}$ is degenerate, as can be seen below.

$$\mathcal{A}_{3}^{d}=\begin{bmatrix}
    0 & 1 & 1 & 1 & 1 & 1 & 1 & 0\\ 0 & 0 & 1 & 0 & 0 & 1 & 0 & 1\\ 0 & 1 & 0 & 0 & 0 & 0 & 1 & 1\\ 0 & 0 & 1 & 0 & 1 & 1 & 0 & 1 \\ 0 & 1 & 0 & 1 & 0 & 0 & 1 & 1\\ 0 & 0 & 0 & 1 & 0 & 0 & 1 & 1 \\ 0 & 0 & 0 & 0 & 1 & 1 & 0 & 1\\ 0 & 0 & 0 & 0 & 0 & 0 & 0 & 0
\end{bmatrix},$$
and has been pointed out above, we will only work with nondegenerate graph algebras coming from the nondegenerate adjacency matrices. Hence, let us take a look at the possible commuting matrices of $\mathcal{A}_{3}^{d}$ and call it $\prescript{c}{}{\mathcal{A}}_{3}^{d}$, and it could be as follows, for $p,q$ some commuting or not commuting numerical parameters.

$$\prescript{c}{}{\mathcal{A}}_{3}^{d}=\begin{bmatrix}
    1 & 0 & 0 & 0 & 0 & 0 & 0 & 0\\ 0 & 0 & 0 & p & 0 & 0 & 1-p & 0\\ 0 & 0 & 0 & 0 & q & 1-q & 0 & 0\\ 0 & 1-p & 0 & 0 & 0 & 0 & p & 0 \\ 0 & 0 & 1-q & 0 & 0 & q & 0 & 0\\ 0 & 0 & q & 0 & 1-q & 0 & 0 & 0 \\ 0 & p & 0 & 1-p & 0 & 0 & 0 & 0\\ 0 & 0 & 0 & 0 & 0 & 0 & 0 & 1
\end{bmatrix},$$

Anyway, it is not too difficult to observe that we have $\mathcal{A}_{3}^{d}\cdot \prescript{c}{}{\mathcal{A}}_{3}^{d} = \prescript{c}{}{\mathcal{A}}_{3}^{d} \cdot \mathcal{A}_{3}^{d}$ only in the case where we have $p=q$. So, $\mathcal{A}_{3}^{d}$ does not possess any quantum symmetries, and for $p\in\{0,1\}$ we will get two different commuting matrices.

 But before considering the $C^*$-graph algebras, we have to find out if the same arguments work for the oriented $Sq_4$. The oriented version of $Sq_4$ will proceed as follows.

%%%%%%%%%%%%
\vspace*{0.5cm}

%\hspace*{2.7cm} 

\begin{tikzpicture}\label{Gra:2}
%[
%box/.style={draw,rectangle,minimum size=2cm,text width=1.5cm,align=left}]
\tikzset{vertex/.style = {shape=circle,draw,minimum size=0.7em}}
\tikzset{edge/.style = {->,> = latex'}}
% vertices
\textcolor{green}{\node[vertex] (a) at  (-4,-4) {$v_5$};}
\textcolor{blue}{\node[vertex] (b) at  (0,4) {$v_4$};}
\textcolor{red}{\node[vertex] (c) at  (8,0) {$v_6$};}
\textcolor{blue}{\node[vertex] (d) at  (4,-8) {$v_1$};}
\textcolor{red}{\node[vertex] (e) at  (-5,5) {$v_7$};}
\textcolor{red}{\node[vertex] (f) at  (-5,-9) {$v_8$};}
\textcolor{green}{\node[vertex] (g) at  (9,5) {$v_9$};}
\textcolor{green}{\node[vertex] (h) at  (9,-9) {$v_{10}$};}
\textcolor{red}{\node[vertex] (h_1) at  (4,4) {$v_{11}$};}
\textcolor{blue}{\node[vertex] (v_3) at  (-4,0) {$v_3$};}
\textcolor{green}{\node[vertex] (h_4) at  (0,-8) {$v_{12}$};}
\textcolor{blue}{\node[vertex] (v_2) at  (8,-4) {$v_2$};}
%\node[vertex] (a1) at (1.5,0) {};
\textcolor{cyan}{\draw[edge] (v_2) to node[left] {\hspace{0cm}\begin{turn}{-360}$e_1$\end{turn}} (c);}
\textcolor{brown}{\draw[edge] (v_2) to node[left] {\hspace{0cm}\begin{turn}{-360}$e_2$\end{turn}} (h);}
\textcolor{blue}{\draw[edge] (v_2) to node[right] {\hspace{0cm}\begin{turn}{-360}$e_3$\end{turn}} (g);}
\textcolor{purple}{\draw[edge] (g) to node[above] {\hspace{0cm}\begin{turn}{-360}$e_4$\end{turn}} (b);}
\textcolor{yellow}{\draw[edge] (g) to node[above] {\hspace{1.7cm}\begin{turn}{-360}$e_{5}$\end{turn}} (e);}
\textcolor{yellow}{\draw[edge] (e) to node[above] {\hspace{-1.7cm}\begin{turn}{-360}$e_6$\end{turn}} (g);}
\textcolor{purple}{\draw[edge] (h) to node[right] {\hspace{0cm}\begin{turn}{-360}$e_{7}$\end{turn}} (c);}
\textcolor{cyan}{\draw[edge] (d) to node[below] {\hspace{0cm}\begin{turn}{-360}$e_{8}$\end{turn}} (f);}
\textcolor{purple}{\draw[edge] (d) to node[below] {\hspace{0cm}\begin{turn}{-360}$e_{9}$\end{turn}} (h_1);}
\textcolor{yellow}{\draw[edge] (h) to node[below] {\hspace{1.7cm}\begin{turn}{-360}$e_{10}$\end{turn}} (f);}
\textcolor{yellow}{\draw[edge] (f) to node[below] {\hspace{-1.7cm}\begin{turn}{-360}$e_{11}$\end{turn}} (h);}
\textcolor{purple}{\draw[edge] (e) to node[left] {\hspace{0cm}\begin{turn}{-360}$e_{12}$\end{turn}} (a);}
\textcolor{teal}{\draw[edge] (f) to node[left] {\hspace{0cm}\begin{turn}{-360}$e_{13}$\end{turn}} (v_3);}
\textcolor{cyan}{\draw[edge] (g) to node[below] {\hspace{0cm}\begin{turn}{-360}$e_{14}$\end{turn}} (v_3);}
\textcolor{teal}{\draw[edge] (h_4) to node[right] {\hspace{0cm}\begin{turn}{-360}$e_{15}$\end{turn}} (b);}
\textcolor{cyan}{\draw[edge] (h_4) to node[right] {\hspace{0cm}\begin{turn}{-360}$e_{16}$\end{turn}} (e);}
\textcolor{cyan}{\draw[edge] (h) to node[right] {\hspace{0cm}\begin{turn}{-360}$e_{17}$\end{turn}} (h_1);}
\textcolor{blue}{\draw[edge] (h) to node[right] {\hspace{0cm}\begin{turn}{-360}$e_{18}$\end{turn}} (v_3);}
\textcolor{yellow}{\draw[edge] (c) to node[below] {\hspace{0cm}\begin{turn}{-360}$e_{19}$\end{turn}} (v_3);}
%\node[vertex] (a2) at (3,0) {};
\textcolor{red}{\draw[edge] (h_1) to node[right] {\hspace{0cm}\begin{turn}{-360}$e_{20}$\end{turn}} (v_3);}
\textcolor{purple}{\draw[edge] (h_4) to node[right] {\hspace{-0.9cm}\begin{turn}{-360}$e_{21}$\end{turn}} (v_3);}
\textcolor{black}{\draw[edge] (h_4) to node[below] {\hspace{0.8cm}\begin{turn}{-360}$e_{22}$\end{turn}} (h_1);}
\textcolor{blue}{\draw[edge] (h_1) to node[above] {\hspace{-0.5cm}\begin{turn}{-360}$e_{23}$\end{turn}} (h_4);}
%edges
\textcolor{cyan}{\draw[edge] (a) to node[above] {\hspace{-0.5cm}\begin{turn}{-360}$e_{24}$\end{turn}} (b);}
\textcolor{purple}{\draw[edge] (v_2) to node[below] {\hspace{0.0cm}\begin{turn}{-360}$e_{25}$\end{turn}} (f);}
\textcolor{teal}{\draw[edge] (v_2) to node[below] {\hspace{0.0cm}\begin{turn}{-360}$e_{26}$\end{turn}} (e);}
\textcolor{yellow}{\draw[edge] (v_2) to node[below] {\hspace{0.0cm}\begin{turn}{-360}$e_{27}$\end{turn}} (a);}
\textcolor{black}{\draw[edge] (v_2) to node[right] {\hspace{0.0cm}\begin{turn}{-360}$e_{28}$\end{turn}} (h_1);}
\textcolor{red}{\draw[edge] (v_2) to node[left] {\hspace{0.0cm}\begin{turn}{-360}$e_{29}$\end{turn}} (h_4);}
%\draw[edge] (b) to node[above] {h} (c);
\textcolor{red}{\draw[edge] (c) to node[below] {\hspace{1.0cm}\begin{turn}{-360}$e_{30}$\end{turn}} (b);}
\textcolor{teal}{\draw[edge] (c) to node[left] {\hspace{0.0cm}\begin{turn}{-360}$e_{31}$\end{turn}} (g);}
\textcolor{brown}{\draw[edge] (a) to node[right] {\hspace{0.0cm}\begin{turn}{-360}$e_{32}$\end{turn}} (v_3);}
\textcolor{black}{\draw[edge] (e) to node[right] {\hspace{0.0cm}\begin{turn}{-360}$e_{33}$\end{turn}} (v_3);}
\textcolor{red}{\draw[edge] (f) to node[right] {\hspace{0cm}\begin{turn}{-360}$e_{34}$\end{turn}} (a);}
%f
\textcolor{brown}{\draw[edge] (e) to node[above] {\hspace{0.0cm}\begin{turn}{-360}$e_{35}$\end{turn}} (b);}
\textcolor{yellow}{\draw[edge] (h_1) to node[below] {\hspace{0.0cm}\begin{turn}{-360}$e_{36}$\end{turn}} (b);}
\textcolor{brown}{\draw[edge] (g) to node[below] {\hspace{0.0cm}\begin{turn}{-360}$e_{37}$\end{turn}} (h_1);}
\textcolor{brown}{\draw[edge] (h_4) to node[above] {\hspace{0.0cm}\begin{turn}{-360}$e_{38}$\end{turn}} (f);}
\textcolor{yellow}{\draw[edge] (d) to node[above] {\hspace{0.0cm}\begin{turn}{-360}$e_{39}$\end{turn}} (h_4);}
\textcolor{teal}{\draw[edge] (d) to node[above] {\hspace{0.0cm}\begin{turn}{-360}$e_{40}$\end{turn}} (h);}
%\draw[edge] (a) to node[below] {f} (d);
\textcolor{brown}{\draw[edge] (d) to node[below] {\hspace{1.6cm}\begin{turn}{20}$e_{41}$\end{turn}} (a);}
\textcolor{black}{\draw[edge] (d) to node[below] {\hspace{0cm}\begin{turn}{20}$e_{42}$\end{turn}} (c);}
%\draw[edge] (g) to node[below] {n} (f);
\textcolor{black}{\draw[edge] (g) to node[above] {\hspace{0.4cm}\begin{turn}{20}$e_{43}$\end{turn}} (f);}
\textcolor{black}{\draw[edge] (f) to node[below] {\hspace{-0.2cm}\begin{turn}{20}$e_{44}$\end{turn}} (g);}
\textcolor{red}{\draw[edge] (d) to node[below] {\hspace{0.2cm}\begin{turn}{20}$e_{45}$\end{turn}} (g);}
\textcolor{blue}{\draw[edge] (d) to node[above] {\hspace{-1.2cm}\begin{turn}{20}$e_{46}$\end{turn}} (e);}
\textcolor{red}{\draw[edge] (e) to node[right] {\hspace{0.2cm}\begin{turn}{20}$e_{47}$\end{turn}} (h);}
\textcolor{red}{\draw[edge] (h) to node[left] {\hspace{-1cm}\begin{turn}{20}$e_{48}$\end{turn}} (e);}
\textcolor{blue}{\draw[edge] (f) to node[below] {\hspace{0.5cm}\begin{turn}{320}$e_{49}$\end{turn}} (b);}
\textcolor{black}{\draw[edge] (h) to node[below] {\hspace{1.2cm}\begin{turn}{20}$e_{50}$\end{turn}} (b);}
%	\draw[edge] (a)  to[bend left] (a1);
%	\draw[edge] (a1) to[bend left] (a);

%	\draw[edge] (a1) to[bend left] (a2);
%	\draw[edge] (a2) to[bend left] (a1);

%	\path (a2) to node {\dots} (c);
%	\node [shape=circle,minimum size=1.5em] (a3) at (4.5,0) {};
%	\draw[edge] (a2) to[bend left] (a3);
%	\draw[edge] (a3) to[bend left] (a2);

%	\node [shape=circle,minimum size=0.7em] (c1) at (6.5,0) {};
%\textcolor{red}{\draw[edge] (b) to node[below] {\hspace{-0.3cm}\begin{turn}{0}$i$\end{turn}} (d);}
%\draw[edge] (b) to node[below] {i} (d);
%\draw[edge] (a) to node[above] {j} (c);
\textcolor{blue}{\draw[edge] (a) to node[above] {\hspace{3cm}\begin{turn}{0}$e_{51}$\end{turn}} (c);}
\textcolor{blue}{\draw[edge] (c) to node[below] {\hspace{-5.5cm}\begin{turn}{0}$e_{52}$\end{turn}} (a);}
%\draw[edge] (b) to[bend left] node[right] {i} (d);
%\draw[edge] (d) to[bend left] node[left] {j} (b);
%\node [below=0.9cm, align=flush center,text width=8cm] at (d)
%{
%	$M_1$
%}

\end{tikzpicture}

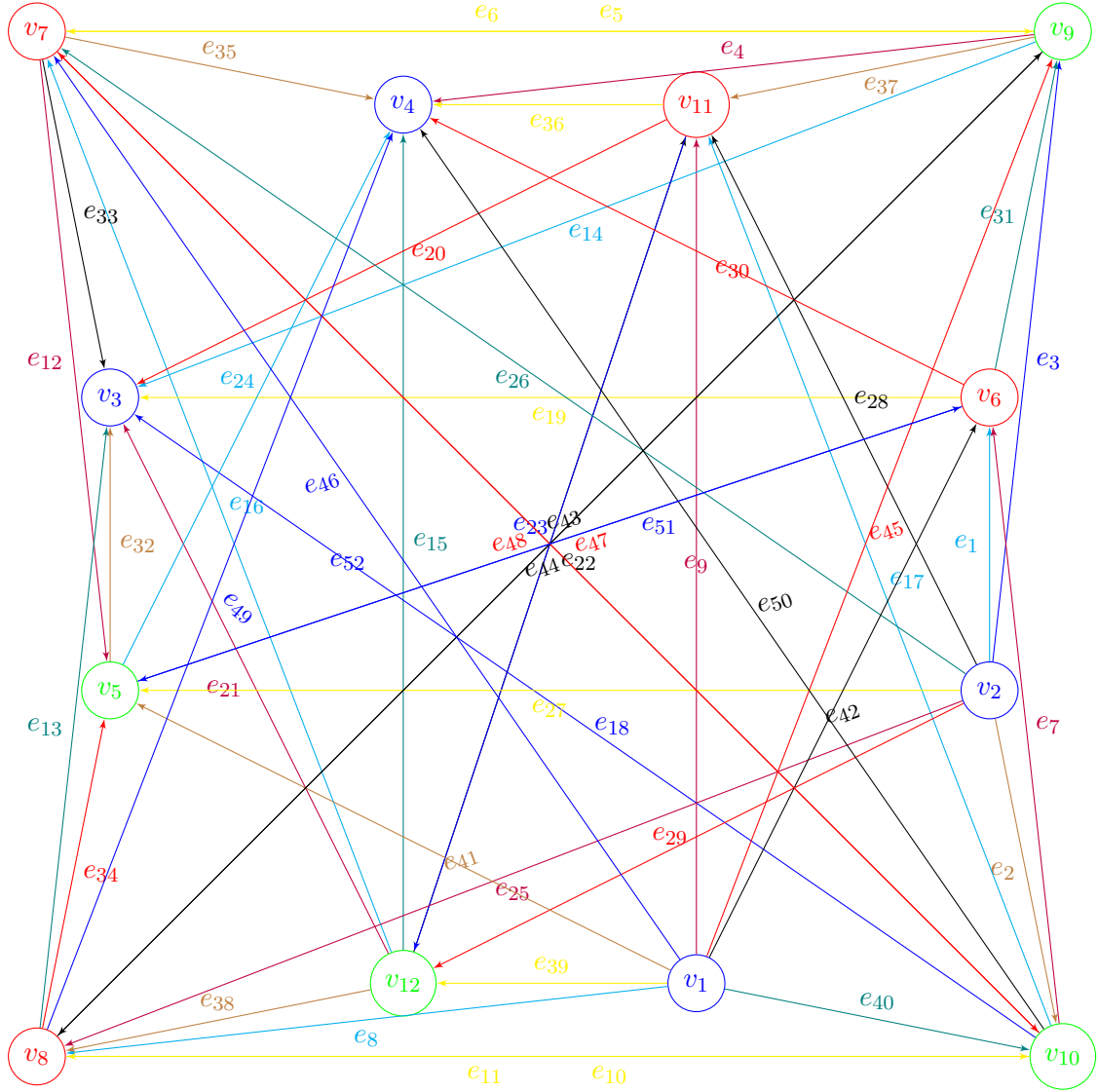
\captionof{figure}{\textbf{Seven connected directed graph $\mathsf{Sq}_{4}^{d}$}}\label{Fig::4}

\vspace*{0.3cm}

%%%%%%%%%%%%%%%%%
\begin{remark}
    \begin{enumerate}[label=\roman*)]
        \item For $n\in\{3,4,\cdots\}$, there are $n-2$ sink and source vertices in $\mathsf{Sq}_{n}^{d}$.
        \item For $n\in\{3,4,\cdots\}$, the number of edges and vertices in $\mathsf{Sq}_{n}^{d}$ is equal to $26(n-2)$, and $4(n-1)$, respectively.
        \item Note that $\mathsf{Sq}_{n}^{d}$ has the vertex chromatic number $\chi_v(\mathsf{Sq}_{n}^{d})=3$, and the edge chromatic number $\chi_e(\mathsf{Sq}_{n}^{d})=2n$, for $n\in\{3,4,\cdots\}$.
    \end{enumerate}
\end{remark}
In the following lemma, we will refer to the sets $\mathscr{V}_{\mathsf{so}}:=\{w,v_2\}$ and $\mathscr{V}_{\mathsf{si}}:=\{v,v_3\}$, the set of source and sink vertices, respectively.
\begin{lemma}\label{Lem::4}
     $\mathsf{Sq}_{4}^{d}$ has 216 Hamiltonian paths.
\end{lemma}
\begin{proof}
 It is not too difficult to observe that the Hamiltonian paths of $\mathsf{Sq}_{4}^{d}$ are as follows. But, as in $\mathsf{Sq}_{4}^{d}$ there are two sink and source vertices (two of each of them); hence, in order to seize the complexity, we divide them into two groups, starting from $w$ (there are 108 Hamiltonian paths in this group) and $v_2$ (there are 108 Hamiltonian paths in this group). Note that in the Hamiltonian paths presented below, they all end with $\to \mathscr{V}_{\mathsf{si}}$, but due to the space limitation, we just omit it. Hence, each Hamiltonian path in the following table presents two Hamiltonian paths.
\begin{align*}
    & v_2\to o\to h_1\to h_4\to l\to n\to m\to u\to k, \quad  w\to o\to h_1\to h_4\to l\to n\to m\to u\to k,\\& v_2\to h_1\to h_4\to l\to n\to m\to o\to k\to u, \quad w\to h_1\to h_4\to l\to n\to m\to o\to k\to u,\\& v_2\to h_1\to h_4\to l\to o\to k\to n\to m\to u, \quad w\to h_1\to h_4\to l\to o\to k\to n\to m\to u,\\& v_2\to l\to o\to m\to n\to h_1\to h_4\to u\to k,  \quad w\to l\to o\to m\to n\to h_1\to h_4\to u\to k,\\& v_2\to m\to n\to l\to o\to h_1\to h_4\to u\to k, \quad w\to m\to n\to l\to o\to h_1\to h_4\to u\to k,\\& v_2\to m\to o\to l\to n\to h_1\to h_4\to u \to k, \quad w\to m\to o\to l\to n\to h_1\to h_4\to u \to k.
\end{align*}

\begin{align*}
&v_2\to u\to k\to n\to m\to o\to h_1\to h_4\to l,\quad w\to u\to k\to n\to m\to o\to h_1\to h_4\to l \\&v_2\to l\to u\to k\to n\to m\to o\to h_1\to h_4,\quad w\to l\to u\to k\to n\to m\to o\to h_1\to h_4,\\&v_2\to h_4\to l\to u\to k\to n\to m\to o\to h_1,\quad w\to h_4\to l\to u\to k\to n\to m\to o\to h_1,\\&v_2\to h_1\to h_4\to l\to u\to k\to n\to m\to o,\quad w\to h_1\to h_4\to l\to u\to k\to n\to m\to o, \\& v_2\to o\to h_1\to h_4\to l\to u\to k\to n\to m,\quad w\to o\to h_1\to h_4\to l\to u\to k\to n\to m,\\&v_2\to m\to o\to h_1\to h_4\to l\to u\to k\to n, \quad w\to m\to o\to h_1\to h_4\to l\to u\to k\to n,\\& v_2\to n\to m\to o\to h_1\to h_4\to l\to u\to k, \quad  w\to n\to m\to o\to h_1\to h_4\to l\to u\to k,\\& v_2\to k\to n\to m \to o\to h_1\to h_4\to l\to u, \quad w\to k\to n\to m \to o\to h_1\to h_4\to l\to u,\\& \quad \\&
\end{align*}

\begin{align*}
&v_2\to l\to u\to k\to n\to h_1\to h_4\to m\to o,\quad w\to l\to u\to k\to n\to h_1\to h_4\to m\to o,\\&v_2\to o\to l\to u\to k\to n\to h_1\to h_4\to m,\quad w\to o\to l\to u\to k\to n\to h_1\to h_4\to m,\\&v_2\to m\to o\to l\to u\to k\to n\to h_1\to h_4,\quad w\to m\to o\to l\to u\to k\to n\to h_1\to h_4,\\&v_2\to h_4\to m\to o\to l\to u\to k\to n\to h_1,\quad w\to h_4\to m\to o\to l\to u\to k\to n\to h_1,\\&v_2\to h_1\to h_4\to m\to o\to l\to u\to k\to n,\quad w\to h_1\to h_4\to m\to o\to l\to u\to k\to n, \\&v_2\to n\to h_1\to h_4\to m\to o\to l\to u\to k,\quad w\to n\to h_1\to h_4\to m\to o\to l\to u\to k,\\&v_2\to k\to n\to h_1\to h_4\to m\to o\to l\to u,\quad w\to k\to n\to h_1\to h_4\to m\to o\to l\to u, \\&v_2\to u\to k\to n\to h_1\to h_4\to m\to o\to l,\quad w\to u\to k\to n\to h_1\to h_4\to m\to o\to l
\end{align*}

\begin{align*}
&v_2\to k\to n\to h_1\to h_4\to l\to o\to m\to u,\quad w\to k\to n\to h_1\to h_4\to l\to o\to m\to u,\\&v_2\to u\to k\to n\to h_1\to h_4\to l\to o\to m,\quad w\to u\to k\to n\to h_1\to h_4\to l\to o\to m,\\&v_2\to m\to u\to k\to n\to h_1\to h_4\to l\to o,\quad w\to m\to u\to k\to n\to h_1\to h_4\to l\to o ,\\&v_2\to o\to m\to u\to k\to n\to h_1\to h_4\to l,\quad w\to o\to m\to u\to k\to n\to h_1\to h_4\to l,\\&v_2\to l\to o\to m\to u\to k\to n\to h_1\to h_4,\quad w\to l\to o\to m\to u\to k\to n\to h_1\to h_4,\\&v_2\to h_4\to l\to o\to m\to u\to k\to n\to h_1,\quad w\to h_4\to m\to u\to k\to n\to l\to o\to h_1,\\&v_2\to h_1\to h_4\to l\to o\to m\to u\to k\to n,\quad w\to h_1\to h_4\to l\to o\to m\to u\to k\to n,\\&v_2\to n\to h_1\to h_4\to l\to o\to m\to u\to k,\quad w\to n\to h_1\to h_4\to l\to o\to m\to u\to k,
\end{align*}

\begin{align*}
&v_2\to k\to o\to m \to n\to h_1\to h_4\to l\to u,\quad w\to k\to o\to m \to n\to h_1\to h_4\to l\to u,\\&v_2\to u\to k\to o\to m \to n\to h_1\to h_4\to l,\quad w\to u\to k\to o\to m \to n\to h_1\to h_4\to l,\\&v_2\to l\to u\to k\to o\to m \to n\to h_1\to h_4,\quad w\to l\to u\to k\to o\to m \to n\to h_1\to h_4,\\&v_2\to h_4\to l\to u\to k\to o\to m \to n\to h_1,\quad w\to h_4\to l\to u\to k\to o\to m \to n\to h_1,\\&v_2\to h_1\to h_4\to l\to u\to k\to o\to m \to n,\quad w\to h_1\to h_4\to l\to u\to k\to o\to m \to n,\\&v_2\to n\to h_1\to h_4\to l\to u\to k\to o\to m,\quad w\to n\to h_1\to h_4\to l\to u\to k\to o\to m,\\&v_2\to m\to n\to h_1\to h_4\to l\to u\to k\to o,\quad w\to m\to n\to h_1\to h_4\to l\to u\to k\to o,\\&v_2\to o\to m\to n\to h_1\to h_4\to l\to u\to k,\quad w\to o\to m\to n\to h_1\to h_4\to l\to u\to k,
\end{align*}

\begin{align*}
&v_2\to l\to o\to h_1\to h_4\to m\to u\to k\to n,\quad w\to l\to o\to h_1\to h_4\to m\to u\to k\to n,\\&v_2\to n\to l\to o\to h_1\to h_4\to m\to u\to k,\quad w\to n\to l\to o\to h_1\to h_4\to m\to u\to k,\\&v_2\to k\to n\to l\to o\to h_1\to h_4\to m\to u,\quad w\to k\to n\to l\to o\to h_1\to h_4\to m\to u,\\&v_2\to u\to k\to n\to l\to o\to h_1\to h_4\to m,\quad w\to u\to k\to n\to l\to o\to h_1\to h_4\to m,\\&v_2\to m\to u\to k\to n\to l\to o\to h_1\to h_4,\quad w\to m\to u\to k\to n\to l\to o\to h_1\to h_4,\\&v_2\to h_4\to m\to u\to k\to n\to l\to o\to h_1,\quad w\to h_4\to m\to u\to k\to n\to l\to o\to h_1,\\& v_2\to h_1\to h_4\to m\to u\to k\to n\to l\to o,\quad w\to h_1\to h_4\to m\to u\to k\to n\to l\to o,\\& v_2\to o\to h_1\to h_4\to m\to u\to k\to n\to l,\quad w\to o\to h_1\to h_4\to m\to u\to k\to n\to l,
\end{align*}

\begin{align*}
& v_2\to o\to l\to n\to h_1\to h_4\to m\to u\to k,\quad w\to o\to l\to n\to h_1\to h_4\to m\to u\to k,\\&v_2\to m\to n\to h_1\to h_4\to l\to o\to k\to u,\quad w\to m\to n\to h_1\to h_4\to l\to o\to k\to u,\\&v_2\to o\to h_1\to h_4\to m\to n\to l\to u\to k,\quad w\to o\to h_1\to h_4\to m\to n\to l\to u\to k,\\&v_2\to m\to o\to k\to n\to h_1\to h_4\to l\to u,\quad w\to m\to o\to k\to n\to h_1\to h_4\to l\to u,\\&v_2\to l\to o\to k\to n\to h_1\to h_4\to m\to u,\quad w\to l\to o\to k\to n\to h_1\to h_4\to m\to u,\\&v_2\to l\to n\to h_1\to h_4\to m\to o\to k\to u,\quad w\to l\to n\to h_1\to h_4\to m\to o\to k\to u,\\&v_2\to h_1\to h_4\to m\to n\to l\to o\to k\to u,\quad w\to h_1\to h_4\to m\to n\to l\to o\to k\to u,\\&v_2\to h_4\to l\to o\to m\to u\to k\to n\to h_1,\quad w\to h_4\to l\to o\to m\to u\to k\to n\to h_1.
\end{align*}

%\begin{align*}
%& v_2\to o\to h_1\to h_4\to m\to n\to l\to u\to k,\quad w\to o\to h_1\to h_4\to m\to n\to l\to u\to k,
%\end{align*}

%\begin{align*}
%& v_2\to m\to o\to k\to n\to h_1\to h_4\to l\to u,\quad w\to m\to o\to k\to n\to h_1\to h_4\to l\to u,
%\end{align*}

%\begin{align*}
%& v_2\to l\to o\to k\to n\to h_1\to h_4\to m\to u,\quad w\to l\to o\to k\to n\to h_1\to h_4\to m\to u,
%\end{align*}

%\begin{align*}
%& v_2\to l\to n\to h_1\to h_4\to m\to o\to k\to u,\quad w\to l\to n\to h_1\to h_4\to m\to o\to k\to u,
%\end{align*}

%\begin{align*}
%& v_2\to h_1\to h_4\to m\to n\to l\to o\to k\to u,\quad w\to h_1\to h_4\to m\to n\to l\to o\to k\to u,
%\end{align*}

%\begin{align*}
%&w\to h_4\to l\to o\to m\to u\to k\to n\to h_1,\quad v_2\to h_4\to l\to o\to m\to u\to k\to n\to h_1
%\end{align*}

\end{proof}

\begin{cl}
    The number of Hamiltonian paths of $\mathsf{Sq}_{n}^{d}$ for $n\in\{3,4,\cdots\}$ is as follows
    \begin{equation}
        \#\mathscr{H}_{\mathsf{Sq}_{n}^{d}}=7(n+1)+188(n-3).
    \end{equation}
\end{cl}

\section{Cuntz-Krieger colored graph $C^*$-algebras}
Before starting this section, let us try to reprove a previously claimed result in \cite[Claim 2.7]{R242} which is necessary to proceed with the results provided in this paper, and note that the original proof of this result could be find in \cite{R243}. For more information concerning $\mathbb K[M_q(n)]$ and the associated directed graphs $\mathcal{G}(\Pi_n)$, we refer the interested reader to \cite{RH24}.

\begin{thm}\label{Cl:Op:1}
   For $\mathcal{G}(\Pi_n)=(\mathcal{G}^0,\mathcal{G}^1)$ the associated directed locally finite graphs with $\mathbb K[M_q(n)]$, and $\Pi_n$ the associated adjacency matrices, and $\mathcal{H}:=\ell^2(\mathbb N)$ the underlying infinite dimensional Hilbert space. The claim is that the set 

    $$S=\{ S_i:=\sum_{j=1}^{\infty}\prescript{i}{}{E}_{\mathcal{E}j-A,(n^2-1)j-D} \mid \ \text{for fixed} \ 1\leq i\leq \frac{(n^3+n^2)(n-1)}{2}\}\label{Equ:PIs},$$

is a Cuntz-Krieger $\mathcal{G}(\Pi_n)$-family for $D\in\{0,\cdots,n^2-2\}$, and $\mathcal{E}$ depends on the degree of the exit edges to the vertex $e_{hk}$, where $i$ is considered as an exit edge, i.e. if $\overset{\rightarrow}{\degg}_{hk}=2$, then we will have $\mathcal{E}=2(n^2-1)$, and if it is 3, then we will have $\mathcal{E}=3(n^2-1)$, and so on, and $A\in\{0,\cdots,\overset{\rightarrow}{\degg}_{hk}\times(n^2-1)\}$, and gives us a graph $C^*$-algebra structure $\mathcal{C}^*(\mathcal{G}(\Pi_n))$.
\end{thm}
\begin{proof}
    In order to prove, we will proceed by induction. But before proceeding, let us clarify the almost approximate values of the parameters $D$. Starting with Hamiltonian paths, we know that they start from the only source vertex $e_{1,1}$ and terminate at the only sink vertex $e_{n,n}$. Hence depending on the first connection, we might have each of $(n^2-1)i-D$, for $D\in\{0,\cdots,n^2-2\}$. For example, if the first connection is $e_{1,1}\rightarrow e_{1,2}$, then we have $D=0$, or if it is $e_{1,1}\rightarrow e_{2,1}$, then we have $D=1$, and \textit{etc}.

    In \cite{R242}, we already have proved that the assertion is true for $n=2$. Bellow, we will prove that it also is true for $n=3$.

    Exactly in a same approach as in \cite[Proposition 3.1]{R242}, we note that there might be finite or infinite dimensional sets of projections. But in order to skip the complication, and to shorten the proof, we skip proving that for some certain vertices $e_{i,j}$, we should have $\dim(P_{e_{i,j}})>\infty$, and based on an already proven fact, all the other projections also should be infinite dimensional. So, let us try to construct an appropriate \textit{CK-}$\mathcal{G}(\Pi_3)$ family satisfying the relations $S_{e}^{*}S_e=P_{r(e)}$, for all edges $e\in\mathcal{G}(\Pi_3)^1$, and $P_{e_{ij}}=\sum\limits_{s(e)=e_{i,j}}S_eS_{e}^{*}$ for the case when $e_{i,j}\in\mathcal{G}(\Pi_3)^0$ is not a sink. 
    
    \hspace*{1cm} \begin{tikzpicture}
%[
%box/.style={draw,rectangle,minimum size=2cm,text width=1.5cm,align=left}]
\tikzset{vertex/.style = {shape=circle,draw,minimum size=0.2em}}
\tikzset{edge/.style = {->,> = latex'}}
% vertices
\node[vertex] (a) at  (-1,1) {$e_{21}$};
\node[vertex] (b) at  (4,3) {$e_{22}$};
\node[vertex] (c) at  (9,1) {$e_{23}$};
\node[vertex] (d) at  (4,-3) {$e_{31}$};
\node[vertex] (e) at  (-2,-4) {$e_{11}$};
\node[vertex] (f) at  (10,-4) {$e_{33}$};
\node[vertex] (g) at  (0.7,0.5) {$e_{12}$};
\node[vertex] (h) at  (2.3,0.5) {$e_{13}$};
\node[vertex] (i) at  (7.3,-1.2) {$e_{32}$};
%\node[vertex] (a1) at (1.5,0) {};
%\node[vertex] (a2) at (3,0) {};
%edges
%\draw[edge] (a) to node[above] {e} (b);
\draw[edge] (a) to node[above] {$j_5$} (b);
\draw[edge] (b) to node[above] {\hspace{-0.9cm}\begin{turn}{-20}$g_6=f_6$\end{turn}} (c);
\draw[edge] (a) to node[below] {\hspace{0.83cm}\begin{turn}{-220}$e_3=f_3$\end{turn}} (d);
\draw[edge] (e) to node[left] {\begin{turn}{-280}$g_1=h_1$\end{turn}} (a);
\draw[edge] (c) to node[right] {\hspace{-0.4cm} \begin{turn}{-80}$i_8=e_8=j_8=h_8$\end{turn}} (f);
\draw[edge] (d) to node[below] {$1$} (f);
\draw[edge] (e) to node[below] {$i_1$} (d);
\draw[edge] (g) to node[above] {\textcolor{red}{$g_3=h_3$}} (h);
\draw[edge] (h) to node[above] {$14$} (f);
\draw[edge] (g) to node[below] {$13$} (i);
\draw[edge] (h) to node[below] {$7$} (c);
\draw[edge] (b) to node[above] {$e_6=h_6=i_6=j_6$} (i);
\draw[edge] (e) to node[above] {\begin{turn}{-296}$e_1=f_1$\end{turn}} (g);
\draw[edge] (e) to node[below] {$j_1$} (h);
\draw[edge] (g) to node[above] {$i_5$} (b);
\draw[edge] (a) to node[below] {$12$} (c);
\draw[edge] (i) to node[above] {$g_8=f_8$} (f);
%\draw[edge] (c) to node[above] {$g_7=f_7$} (i);
%	\draw[edge] (a)  to[bend left] (a1);
%	\draw[edge] (a1) to[bend left] (a);

%	\draw[edge] (a1) to[bend left] (a2);
%	\draw[edge] (a2) to[bend left] (a1);

%	\path (a2) to node {\dots} (c);
%	\node [shape=circle,minimum size=1.5em] (a3) at (4.5,0) {};
%	\draw[edge] (a2) to[bend left] (a3);
%	\draw[edge] (a3) to[bend left] (a2);

%	\node [shape=circle,minimum size=0.7em] (c1) at (6.5,0) {};
\draw[edge] (b) to[bend left] node[left] {\textcolor{red}{$10$}} (d);
\draw[edge] (d) to[bend left] node[right] {$g_5=h_5$} (b);
\draw[edge] (h) to[bend left] node[left] {\textcolor{RoyalBlue}{\hspace{-3cm}\begin{turn}{-260}$g_4=h_4=j_2$\end{turn}}} (d);
\draw[edge] (d) to[bend left] node[right] {\hspace{0.5cm}\begin{turn}{-250}$e_4=f_4=i_2$\end{turn}} (h);
\draw[edge] (h) to[bend left] node[below] {$i_3$} (a);
\draw[edge] (a) to[bend left] node[above] {$8$} (h);
\draw[edge] (a) to[bend left] node[above] {\textcolor{RoyalBlue}{$g_2=h_2=i_4$}} (g);
\draw[edge] (g) to[bend left] node[below] {\textcolor{red}{$e_2=f_2=j_4$}} (a);
\draw[edge] (b) to[bend left] node[below] {\textcolor{red}{$9$}} (h);
\draw[edge] (h) to[bend left] node[right] {\textcolor{red}{\hspace{-0.5cm}\begin{turn}{50}$e_5=f_5$\end{turn}}} (b);
\draw[edge] (i) to[bend left] node[above] {\textcolor{Red}{\hspace{-1cm}\begin{turn}{0}$e_7=h_7=i_7=j_7$\end{turn}}} (c);
\draw[edge] (c) to[bend left] node[below] {\textcolor{Red}{\vspace{-1cm}\begin{turn}{-320}$g_7=f_7$\end{turn}}} (i);
\draw[edge] (h) to[bend left] node[below] {\textcolor{RoyalBlue}{$2$}} (i);
\draw[edge] (i) to[bend left] node[above] {\textcolor{RoyalBlue}{$3$}} (h);
\draw[edge] (c) to[bend left] node[below] {\textcolor{RoyalBlue}{$5$}} (d);
\draw[edge] (d) to[bend left] node[below] {\textcolor{RoyalBlue}{$4$}} (c);
\draw[edge] (d) to node[below] {\textcolor{Red}{$11$}} (i);
\draw[edge] (g) to[bend right] node[below] {\textcolor{Red}{$6$}} (d);
\draw[edge] (d) to[bend right] node[above] {\textcolor{Red}{$j_3$}} (g);
%\node [below=0.9cm, align=flush center,text width=8cm] at (d)
%{
%	$M_1$
%}

\end{tikzpicture}
\vspace*{0.2cm}

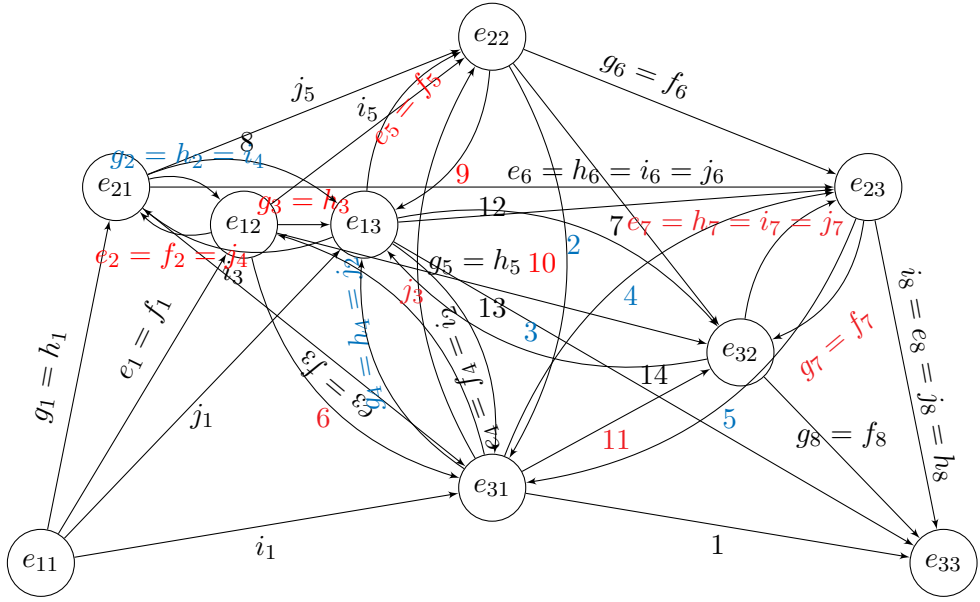
\captionof{figure}{\textbf{Directed locally connected graph $\mathcal{G}(\Pi_3)$}}

\vspace*{0.5cm}
    
 We note that $\mathcal{G}(\Pi_3)$ consists of $6$ Hamiltonian paths listed below:

 \begin{align*}
     &\mathcal{H}_1: e_{11}\overset{e_1}{\longrightarrow}e_{12}\overset{e_2}{\longrightarrow}e_{21}\overset{e_3}{\longrightarrow}e_{31}\overset{e_4}{\longrightarrow}e_{13}\overset{e_5}{\longrightarrow}e_{22}\overset{e_6}{\longrightarrow}e_{32}\overset{e_7}{\longrightarrow}e_{23}\overset{e_8}{\longrightarrow}e_{33},\\&\mathcal{H}_2:e_{11}\overset{f_1}{\longrightarrow}e_{12}\overset{f_2}{\longrightarrow}e_{21}\overset{f_3}{\longrightarrow}e_{31}\overset{f_4}{\longrightarrow}e_{13}\overset{f_5}{\longrightarrow}e_{22}\overset{f_6}{\longrightarrow}e_{23}\overset{f_7}{\longrightarrow}e_{32}\overset{f_8}{\longrightarrow}e_{33},\\&\mathcal{H}_3:e_{11}\overset{g_1}{\longrightarrow}e_{21}\overset{g_2}{\longrightarrow}e_{12}\overset{g_3}{\longrightarrow}e_{13}\overset{g_4}{\longrightarrow}e_{31}\overset{g_5}{\longrightarrow}e_{22}\overset{g_6}{\longrightarrow}e_{23}\overset{g_7}{\longrightarrow}e_{32}\overset{g_8}{\longrightarrow}e_{33},\\&\mathcal{H}_4:e_{11}\overset{h_1}{\longrightarrow}e_{21}\overset{h_2}{\longrightarrow}e_{12}\overset{h_3}{\longrightarrow}e_{13}\overset{h_4}{\longrightarrow}e_{31}\overset{h_5}{\longrightarrow}e_{22}\overset{h_6}{\longrightarrow}e_{32}\overset{h_7}{\longrightarrow}e_{23}\overset{h_8}{\longrightarrow}e_{33},\\&\mathcal{H}_5:e_{11}\overset{i_1}{\longrightarrow}e_{31}\overset{i_2}{\longrightarrow}e_{13}\overset{i_3}{\longrightarrow}e_{21}\overset{i_4}{\longrightarrow}e_{12}\overset{i_5}{\longrightarrow}e_{22}\overset{i_6}{\longrightarrow}e_{32}\overset{i_7}{\longrightarrow}e_{23}\overset{i_8}{\longrightarrow}e_{33},\\&\mathcal{H}_6:e_{11}\overset{j_1}{\longrightarrow}e_{13}\overset{j_2}{\longrightarrow}e_{31}\overset{j_3}{\longrightarrow}e_{12}\overset{j_4}{\longrightarrow}e_{21}\overset{j_5}{\longrightarrow}e_{22}\overset{j_6}{\longrightarrow}e_{32}\overset{j_7}{\longrightarrow}e_{23}\overset{j_8}{\longrightarrow}e_{33}.
 \end{align*}

    Now the claim is that the desired set of partial isometries will be as follows:
    \begin{align*}
        &S_{e_1}=S_{f_1}=\sum_{j=1}^{\infty}E_{32j-31,8j}, S_{g_1}=S_{h_1}=\sum_{j=1}^{\infty}E_{32j-23,8j-1}, S_{i_1}=\sum_{j=1}^{\infty}E_{32j-7,8j-2},\\& S_{j_1}=\sum_{j=1}^{\infty}E_{32j-15,8j-3},\\&S_{g_2}=S_{h_2}=S_{i_4}=\sum_{j=1}^{\infty}E_{40j-25,8j}, S_{e_2}=S_{f_2}=S_{j_4}=\sum_{j=1}^{\infty}E_{40j-32,8j-1}, \\& S_{j_2}=\sum_{j=1}^{\infty}E_{48j-35,8j-2}, S_{i_2}=\sum_{j=1}^{\infty}E_{48j-42,8j-3}, \\&S_{j_3}=\sum_{j=1}^{\infty}E_{48j-10,8j}, S_{i_3}=\sum_{j=1}^{\infty}E_{48j-27,8j-1}, S_{e_3}=S_{f_3}=\sum_{j=1}^{\infty}E_{40j-33,8j-2}, 
    \end{align*}

    \begin{align*}
    &S_{g_3}=S_{h_3}=\sum_{j=1}^{\infty}E_{40j-24,8j-3}, \\&S_{i_4}=\sum_{j=1}^{\infty}E_{40j-25,8j}, S_{j_4}=\sum_{j=1}^{\infty}E_{40j-32,8j-1}, S_{g_4}=S_{h_4}=S_{j_2}=\sum_{j=1}^{\infty}E_{48j-35,8j-2}, \\&S_{e_4}=S_{f_4}=S_{i_2}=\sum_{j=1}^{\infty}E_{48j-42,8j-3},\\&S_{e_5}=S_{f_5}=\sum_{j=1}^{\infty}E_{48j-43,8j-4}, S_{g_5}=S_{h_5}=\sum_{j=1}^{\infty}E_{48j-34,8j-4}, S_{i_5}=\sum_{j=1}^{\infty}E_{40j-8,8j-4}, \\&S_{j_5}=\sum_{j=1}^{\infty}E_{40j-1,8j-4},\\&S_{e_6}=S_{h_6}=S_{i_6}=S_{j_6}=\sum_{j=1}^{\infty}E_{32j-28,8j-5}, S_{g_6}=S_{f_6}=\sum_{j=1}^{\infty}E_{24j-14,8j-5}, \\& S_{i_6}=S_{j_6}=\sum_{j=1}^{\infty}E_{32j-28,8j-5},
        \\&\hspace{0cm} S_{e_7}=S_{h_7}=S_{i_7}=S_{j_7}=\sum_{j=1}^{\infty}E_{24j-21,8j-6}, S_{g_7}=S_{f_7}=\sum_{j=1}^{\infty}E_{32j-20,8j-6},\\& \hspace{0cm} S_{i_7}=\sum_{j=1}^{\infty}E_{24j-21,8j-6}, S_{j_7}=\sum_{j=1}^{\infty}E_{24j-21,8j-6},\\&\hspace{0cm} S_{e_8}=S_{h_8}=S_{i_8}=S_{j_8}=\sum_{j=1}^{\infty}E_{24j-22,8j-7}, S_{g_8}=S_{f_8}=\sum_{j=1}^{\infty}E_{24j-13,8j-7}, \\&\hspace{0cm} S_{i_8}=\sum_{j=1}^{\infty}E_{24j-22,8j-7}, S_{j_8}=\sum_{j=1}^{\infty}E_{24j-22,8j-7},
    \end{align*}
    and the rest of the isometries are apart from the set of Hamiltonian paths, and could be outlined as follows:
    \begin{align*}
&S_{1}=\sum_{j=1}^{\infty}E_{48j-2,8j-7}, S_{14}=\sum_{j=1}^{\infty}E_{48j-11,8j-7}, S_{4}=\sum_{j=1}^{\infty}E_{48j-29,8j-7},\\& S_{7}=\sum_{j=1}^{\infty}E_{48j-19,8j-6}, S_{12}=\sum_{j=1}^{\infty}E_{40j-9,8j-6},\\& S_{2}=\sum_{j=1}^{\infty}E_{48j-3,8j-5}, S_{11}=\sum_{j=1}^{\infty}E_{48j-26,8j-5}, S_{13}=\sum_{j=1}^{\infty}E_{40j-11,8j-5}, 
    \end{align*}
    \begin{align*}
        & S_{3}=\sum_{j=1}^{\infty}E_{24j-5,8j-3}, S_{8}=\sum_{j=1}^{\infty}E_{40j-17,8j-3}, S_{9}=\sum_{j=1}^{\infty}E_{32j-12,8j-3}, \\& S_{5}=\sum_{j=1}^{\infty}E_{24j-6,8j-2}, S_{6}=\sum_{j=1}^{\infty}E_{40j-16,8j-2}, S_{5}=\sum_{j=1}^{\infty}E_{32j-4,8j-2}.
    \end{align*}
    Now the goal is to prove that the above isometries satisfy in the relations of Cuntz-Krieger $\mathcal{G}(\Pi_3)$-family, i.e. $S_{e}^{*}S_e=P_{r(e)}$, for all edges $e\in\mathcal{G}(\Pi_3)^1$, and $P_{e_{ij}}=\sum\limits_{s(e)=e_{i,j}}S_eS_{e}^{*}$ for the case when $e_{i,j}\in\mathcal{G}(\Pi_3)^0$ is not a sink, and this could be verified simply. Here we just put it over as follows.

    \begin{align*}
        &P_{e_{13}}=S_{3}^{*}S_3=S_{g_3}^{*}S_{g_3}=S_{j_1}^{*}S_{j_1}=S_{e_4}^{*}S_{e_4}=S_{9}^{*}S_{9}=S_{8}^{*}S_{8}=S_{e_5}S_{e_5}^{*}+S_{g_4}S_{g_4}^{*}+S_{i_3}S_{i_3}^{*}\\& \hspace{0.7cm}+S_{7}S_{7}^{*}+S_{14}S_{14}^{*}+S_{2}S_{2}^{*},\\&P_{e_{22}}=S_{e_5}^{*}S_{e_5}=S_{g_5}^{*}S_{g_5}=S_{j_5}^{*}S_{j_5}=S_{i_5}^{*}S_{i_5}=S_{e_6}S_{e_6}^{*}+S_{g_6}S_{g_6}^{*}+S_{9}S_{9}^{*}+S_{10}S_{10}^{*},\\&P_{e_{23}}=S_{e_7}^{*}S_{e_7}=S_{12}^{*}S_{12}=S_{g_6}^{*}S_{g_6}=S_{7}^{*}S_{7}=S_{4}^{*}S_{4}=S_{e_8}S_{e_8}^{*}+S_{g_7}S_{g_7}^{*}+S_{5}S_{5}^{*}, \\&P_{e_{32}}=S_{2}^{*}S_{2}=S_{g_7}^{*}S_{g_7}=S_{e_6}^{*}S_{e_6}=S_{13}^{*}S_{13}=S_{11}^{*}S_{11}=S_{e_7}S_{e_7}^{*}+S_{g_8}S_{g_8}^{*}+S_{3}S_{3}^{*},
        \end{align*}
        
        \begin{align*}
   &P_{e_{33}}=S_{1}^{*}S_{1}=S_{14}^{*}S_{14}=S_{e_8}^{*}S_{e_8}=S_{g_8}^{*}S_{g_8},\\&P_{e_{31}}=S_{i_1}^{*}S_{i_1}=S_{e_3}^{*}S_{e_3}=S_{10}^{*}S_{10}=S_{g_4}^{*}S_{g_4}=S_{5}^{*}S_{5}=S_{6}^{*}S_{6}=S_{e_4}S_{e_4}^{*}+S_{g_5}S_{g_5}^{*}+S_{11}S_{11}^{*}\\&\hspace{0.7cm}+S_{4}S_{4}^{*}+S_{j_3}S_{j_3}^{*}+S_{1}S_{1}^{*},\\&P_{e_{11}}=S_{e_1}^{*}S_{e_1}+S_{g_1}S_{g_1}^{*}+S_{i_1}S_{i_1}^{*}+S_{j_1}S_{j_1}^{*},\\&P_{e_{21}}=S_{g_1}^{*}S_{g_1}=S_{e_2}^{*}S_{e_2}=S_{i_3}^{*}S_{i_3}=S_{e_3}S_{e_3}^{*}+S_{g_2}S_{g_2}^{*}+S_{8}S_{8}^{*}+S_{12}S_{12}^{*}+S_{j_5}S_{j_5}^{*},\\&P_{e_{12}}=S_{g_2}^{*}S_{g_2}=S_{e_1}^{*}S_{e_1}=S_{j_3}^{*}S_{j_3}=S_{e_2}S_{e_2}^{*}+S_{g_3}S_{g_3}^{*}+S_{6}S_{6}^{*}+S_{i_5}S_{i_5}^{*}+S_{13}S_{13}^{*},
    \end{align*}
and at this point, we are not going to have a detailed proof of the above relations, since it is easy to see them. This provide us with the desired Cuntz-Krieger $\mathcal{G}(\Pi_3)$-family, and shows that the $C^*(S,P)$ is an infinite-dimensional $C^*$-algebra.

Now in order to proceed further, let us see what we have. For $n=2$, we have $S_i=\sum_{j=1}^{\infty}\prescript{i}{}{E}_{\mathcal{E}j-A,3j-D}$, and for $n=3$ case we have $S_i=\sum_{j=1}^{\infty}\prescript{i}{}{E}_{\mathcal{E}j-A,8j-D}$, and $S_i=\sum_{j=1}^{\infty}\prescript{i}{}{E}_{\mathcal{E}j-A,15j-D}$ for the $n=4$ case, and so on. So, we get a sequence of numbers $3, 8, 15, 24, 35, \cdots$. In order to have a defining rule for this sequence, we use the recurrence relations $a_{n-2}+2n=a_{n-1}$, and $b_{n-2}+1=b_{n-1}$, for $n\in\{2,3,\cdots\}$, together with conditional relations $a_0=1$ and $b_0=2$. 

Now consider the recurrence relation $h_n+2n+1=h_{n+1}$, for $h_n=b_{n-2}+a_{n-2}$ and the conditional relation $h_2=b_0+a_0=3$.

Hence, we get the following relation
$$S=\{ S_i:=\sum_{j=1}^{\infty}\prescript{i}{}{E}_{\mathcal{E}j-A,h_nj-D} \mid \ \text{for fixed} \ 1\leq i\leq \frac{(n^3+n^2)(n-1)}{2}\}\label{Equ:PIs:::},$$
for $h_{n}$ as above,and $n\in\{2,\cdots\}$. Now we can proceed by induction on $n$. For $n=2,3$, we already have seen that the assertion holds. Suppose that for $\mathcal{G}(\Pi_i)$, $i\in\{2,3,\cdots\}$ (\ref{Equ:PIs:::}) satisfies.

Now the claim is that for $\mathcal{G}(\Pi_{n+1})$ we have
$$S=\{ S_i:=\sum_{j=1}^{\infty}\prescript{i}{}{E}_{\mathcal{E}j-A,h_{n+1}j-D} \mid \ \text{for fixed} \ 1\leq i\leq \frac{(n^3+n^2)(n-1)}{2}\}\label{Equ:PIs:::},$$
which is almost clear, since we have $(n+1)^2-1=n^2-1+2n+1=h_n+2n+1=h_{n+1},$ and we are done with the induction steps!

%    Now suppose that the assertion is true for $\mathcal{G}(\Pi_n)$. We will prove that it is true for $\mathcal{G}(\Pi_{n+1})$.

 %   $\mathcal{G}(\Pi_n)$ has $4n-6$ Hamiltonian paths with edge degree $n^2-1$, starting from $e_{1,1}$ and terminated at $e_{n,n}$. In other words $\mathcal{G}(\Pi_{n+1})$ has $2(\mathcal{D}_{n+1}-1)=2(2((n+1)-1)-1=2(2n-1)=4n-2$ Hamiltonian paths, each with edge degree $(n+1)^2-1$, starting from $e_{1,1}$ and terminating at $e_{n+1,n+1}$. And as our approach is based on construction, and based on the induction steps, and the fact that $\mathcal{G}(\Pi_n)\subset \mathcal{G}(\Pi_{n+1})$, hence clearly the relation of $\mathcal{G}(\Pi_{n+1})$-Cuntz-Krieger family will satisfy in the case of $S_i$'s associated with $\mathcal{G}(\Pi_{n+1})$.
\end{proof}

Now the plan is to apply Theorem \ref{Cl:Op:1} to the graphs in $\mathsf{G}_{\mathrm{s}}$, to see if it will provide us with appropriate finite/infinite $C^*$-graph algebras or not. The claim might be extended to the case of an arbitrary colored directed graph as follows.

\begin{cl}\label{Cl:Op:1}
  Let $n\in\{3,4,\cdots\}$, and $\Gamma=(\Gamma^0,\Gamma^1)$ be an arbitrary colored directed graph with chromatic vertex number $\chi_v(\Gamma)$ and $|\Gamma^0|=4n$, with $n-2$ equal number of sink and source vertices. And 
let $\mathcal{A}_n$ be the associated adjacency matrix, and  $\mathcal{H}:=\ell^2(\mathbb N)$ be the underlying infinite dimensional Hilbert space. Then the claim is that the set 

    $$S=\{ S_i:=\sum_{j=1}^{\infty}\prescript{i}{}{E}_{\mathcal{E}j-A,(n^2-1)j-D} \mid \ \text{for fixed} \ 1\leq i\leq 26(n-2)\}\label{Equ:PIs},$$

is a Cuntz-Krieger $\mathcal{G}$-family for $D\in\{0,\cdots,n^2-2\}$, and $\mathcal{E}$ depends on the degree of the exit edges of the vertex $e_{hk}$, where $i$ is considered as an exit edge, i.e. if $\overset{\rightarrow}{\degg}_{hk}=2$, then we will have $\mathcal{E}=2(n^2-1)$, and if it is 3, then we will have $\mathcal{E}=\chi_v(\Gamma)(n^2-1)$, and so on, and $A\in\{0,\cdots,\overset{\rightarrow}{\degg}_{hk}\times(n^2-1)\}$. 

$S$ gives us an infinite-dimensional graph $C^*$-algebra structure $\mathcal{C}^*(\Gamma)$.

\end{cl}

Let us try to prove the following proposition for the lower dimensional case.

\begin{prop}
    For $\mathsf{Sq}_{3}^{d}$, and its adjacency matrix $\mathcal{A}_3$, let $\mathcal{H}:=\ell^2(\mathbb N)$ be the underlying infinite dimensional Hilbert space. Then the set 
    $$S=\{ S_{e_i}:=\sum_{j=1}^{\infty}\prescript{i}{}{E}_{\mathcal{E}j-A,8j-D} \mid \ \text{for fixed} \ 1\leq i\leq 26\}\label{Equ:PIs},$$
    is a Cuntz-Krieger $\mathcal{A}_3$-family for $D\in\{0,\cdots,7\}$,  $\mathcal{E}\in\{16, 24, 32, 48\}$, and $A\in\{0,\cdots,\mathcal{E}\}$ depending on the chosen edges, and gives us the graph $C^*$-algebra structure $\mathcal{C}^*(\mathsf{Sq}_{3}^{d})$.
\end{prop}
\begin{proof}
In order to prove this result, we will follow the same assertions as in \cite[Proposition 3.1]{R242} and \cite[Theorem 4.21]{R243}, but to avoid any possible complexity, we will omit some details. The idea is to associate each edge with a partial isometry, and to each vertex an orthogonal projection so that the relations of Definition \ref{Defn::1...} are satisfied, and there could be sets of projections of finite or infinite dimension. Hence, the first step is to clarify which of these spaces could be our living space!

For $\{\mathscr{S},\mathscr{P}\}\in B(\mathcal{H}):=B(\ell^2(\mathbb N))$, we begin by looking at the dimensions of the subspaces $P_{v_k}\mathcal{H}$ for $k\in\{1,\cdots,8\}$ to see whether we need to consider a finite-dimensional set of projections or not. We have

\begin{align}
&\dim(P_{v_3}\mathcal{H})=\dim(S_{e_{25}}\mathcal{H})+\dim(S_{e_{16}}\mathcal{H})+\dim(S_{e_{8}}\mathcal{H})=\dim(P_{v_{6}}\mathcal{H})+\dim(P_{v_{8}}\mathcal{H})+\dim(P_{v_{1}}\mathcal{H})\notag \\& \dim(P_{v_4}\mathcal{H})=\dim(S_{e_{26}}\mathcal{H})+\dim(S_{e_{15}}\mathcal{H})+\dim(S_{e_{10}}\mathcal{H})=\dim(P_{v_{3}}\mathcal{H})+\dim(P_{v_{1}}\mathcal{H})+\dim(P_{v_{5}}\mathcal{H})\notag \\& \dim(P_{v_5}\mathcal{H})=\dim(S_{e_{2}}\mathcal{H})+\dim(S_{e_{13}}\mathcal{H})+\dim(S_{e_{18}}\mathcal{H})=\dim(P_{v_{1}}\mathcal{H})+\dim(P_{v_{6}}\mathcal{H})+\dim(P_{v_{7}}\mathcal{H})\notag \\& \dim(P_{v_6}\mathcal{H})=\dim(S_{e_{3}}\mathcal{H})+\dim(S_{e_{21}}\mathcal{H})+\dim(S_{e_{14}}\mathcal{H})=\dim(P_{v_{1}}\mathcal{H})+\dim(P_{v_{8}}\mathcal{H})+\dim(P_{v_{5}}\mathcal{H})\notag \\& \dim(P_{v_7}\mathcal{H})=\dim(S_{e_{5}}\mathcal{H})+\dim(S_{e_{6}}\mathcal{H})+\dim(S_{e_{7}}\mathcal{H})+\dim(S_{e_{17}}\mathcal{H})\label{Pro::.1} \\&\hspace{1.9cm}=\dim(P_{v_{1}}\mathcal{H})+\dim(P_{v_{3}}\mathcal{H})+\dim(P_{v_{5}}\mathcal{H})+\dim(P_{v_{8}}\mathcal{H})\notag \\& \dim(P_{v_8}\mathcal{H})=\dim(S_{e_{1}}\mathcal{H})+\dim(S_{e_{4}}\mathcal{H})+\dim(S_{e_{9}}\mathcal{H})+\dim(S_{e_{20}}\mathcal{H})\notag\\& \hspace{1.9cm}=\dim(P_{v_{1}}\mathcal{H})+\dim(P_{v_{4}}\mathcal{H})+\dim(P_{v_{6}}\mathcal{H})+\dim(P_{v_{7}}\mathcal{H})\notag \\& \dim(P_{v_2}\mathcal{H})=\dim(S_{e_{1}}\mathcal{H})+\dim(S_{e_{6}}\mathcal{H})+\dim(S_{e_{11}}\mathcal{H})+\dim(S_{e_{12}}\mathcal{H})+\dim(S_{e_{23}}\mathcal{H})+\dim(S_{e_{24}}\mathcal{H})\notag\\& \hspace{1.9cm}=\dim(P_{v_{3}}\mathcal{H})+\dim(P_{v_{4}}\mathcal{H})+\dim(P_{v_{5}}\mathcal{H})+\dim(P_{v_{6}}\mathcal{H})+\dim(P_{v_{7}}\mathcal{H})+\dim(P_{v_{8}}\mathcal{H}).\notag
\end{align}

From equations \ref{Pro::.1} one can conclude that the projections involved in the set $\mathscr{P}$ associated with $\mathsf{Sq}_{3}^{d}$ cannot be finite dimensional and we should have $\dim(P_{v_k})>\infty$ for $k\in\{1,\cdots,8\}$. So, let us try to construct the appropriate set of \textit{CK}-$\mathsf{Sq}_{3}^{d}$ family satisfying in the relation \ref{Equ:KC}, which are as follows.

\begin{align*}
&S_{e_1}=\sum\limits_{j=1}^{\infty}\prescript{1}{}{E}_{16j-13,8j}, \quad \quad \ S_{e_2}=\sum\limits_{j=1}^{\infty}\prescript{2}{}{E}_{32j-3,8j-4}, \quad \ \ \ \ S_{e_3}=\sum\limits_{j=1}^{\infty}\prescript{3}{}{E}_{32j-3,8j-3},\\& S_{e_4}=\sum\limits_{j=1}^{\infty}\prescript{4}{}{E}_{24j-18,8j-1}, \quad \ S_{e_5}=\sum\limits_{j=1}^{\infty}\prescript{5}{}{E}_{24j-17,8j-2}, \quad \ \ \  S_{e_6}=\sum\limits_{j=1}^{\infty}\prescript{6}{}{E}_{48j-46,8j}, \\& S_{e_7}=\sum\limits_{j=1}^{\infty}\prescript{7}{}{E}_{48j-46,8j}, \quad \quad \ S_{e_8}=\sum\limits_{j=1}^{\infty}\prescript{8}{}{E}_{24j-19,8j-6}, \quad \ \ \ S_{e_9}=\sum\limits_{j=1}^{\infty}\prescript{9}{}{E}_{24j-13,8j-1}, \\& S_{e_{10}}=\sum\limits_{j=1}^{\infty}\prescript{10}{}{E}_{24j-4,8j-5}, \quad \  S_{e_{11}}=\sum\limits_{j=1}^{\infty}\prescript{11}{}{E}_{24j-9,8j}, \quad \ \ \ \ S_{e_{12}}=\sum\limits_{j=1}^{\infty}\prescript{12}{}{E}_{24j-10,8j}, \\& S_{e_{13}}=\sum\limits_{j=1}^{\infty}\prescript{13}{}{E}_{48j-47,8j-4}, \quad S_{e_{14}}=\sum\limits_{j=1}^{\infty}\prescript{14}{}{E}_{48j-46,8j-3}, \quad S_{e_{15}}=\sum\limits_{j=1}^{\infty}\prescript{15}{}{E}_{48j-45,8j-5}, \\& S_{e_{16}}=\sum\limits_{j=1}^{\infty}\prescript{16}{}{E}_{48j-44,8j-6}, \quad S_{e_{17}}=\sum\limits_{j=1}^{\infty}\prescript{17}{}{E}_{32j-4,8j-2}, \quad \  S_{e_{18}}=\sum\limits_{j=1}^{\infty}\prescript{18}{}{E}_{24j-2,8j-4}, 
\\& S_{e_{19}}=\sum\limits_{j=1}^{\infty}\prescript{19}{}{E}_{48j-43,8j-2}, \quad S_{e_{20}}=\sum\limits_{j=1}^{\infty}\prescript{20}{}{E}_{48j-42,8j-1}, \quad S_{e_{21}}=\sum\limits_{j=1}^{\infty}\prescript{21}{}{E}_{24j-1,8j-3},\\& S_{e_{22}}=\sum\limits_{j=1}^{\infty}\prescript{22}{}{E}_{32j-9,8j-1}, \quad \ S_{e_{23}}=\sum\limits_{j=1}^{\infty}\prescript{23}{}{E}_{32j-20,8j}, \quad \ \ \ S_{e_{24}}=\sum\limits_{j=1}^{\infty}\prescript{24}{}{E}_{32j-19,8j}, \\& S_{e_{25}}=\sum\limits_{j=1}^{\infty}\prescript{25}{}{E}_{24j-5,8j-6}, \quad S_{e_{26}}=\sum\limits_{j=1}^{\infty}\prescript{26}{}{E}_{24j-6,8j-5}.
\end{align*}

The above set of \textit{CK}-$\mathsf{Sq}_{3}^{d}$ family $\mathscr{S}=\{S_{e_{\ell}} \mid \ell \in \{1,\cdots,26\}\}$ provides us with an infinite dimensional $C^*$-graph algebra structure $C^*(\mathsf{Sq}_{3}^{d})$.
\end{proof}

But we have been more looking for a finite dimensional $C^*$-graph algebra structure!
\subsection{Looking for a solution!}
Let us invite back the commuting set of matrices $\prescript{c}{}{\mathcal{A}}_{n}^{d}$ to the scene, once again. Starting from $\prescript{c}{}{\mathcal{A}}_{3}^{d}$ and noting that the only choices for $p$ and $q$ are when we have $p=q=1$, and in this case we have the following matrix.

$$\prescript{c}{}{\mathcal{A}}_{3}^{d}=\begin{bmatrix}
    1 & 0 & 0 & 0 & 0 & 0 & 0 & 0\\ 0 & 0 & 0 & 1 & 0 & 0 & 0 & 0\\ 0 & 0 & 0 & 0 & 1 & 0 & 0 & 0\\ 0 & 0 & 0 & 0 & 0 & 0 & 1 & 0 \\ 0 & 0 & 0 & 0 & 0 & 1 & 0 & 0\\ 0 & 0 & 1 & 0 & 0 & 0 & 0 & 0 \\ 0 & 1 & 0 & 0 & 0 & 0 & 0 & 0\\ 0 & 0 & 0 & 0 & 0 & 0 & 0 & 1
\end{bmatrix},$$
and the associated directed graph will be as follows,

\vspace*{0.3cm}

\hspace*{2.8cm} \begin{tikzpicture}\label{Gra:2}
%[
%box/.style={draw,rectangle,minimum size=2cm,text width=1.5cm,align=left}]
\tikzset{vertex/.style = {shape=circle,draw,minimum size=0.7em}}
\tikzset{edge/.style = {->,> = latex'}}
% vertices
\textcolor{green}{\node[vertex] (a) at  (5,0) {$v_4$};}
\textcolor{blue}{\node[vertex] (b) at  (4,2) {$v_2$};}
\textcolor{red}{\node[vertex] (c) at  (3,0) {$v_3$};}
\textcolor{blue}{\node[vertex] (d) at  (4,-2) {$v_1$};}
\textcolor{red}{\node[vertex] (e) at  (0,3) {$v_8$};}
\textcolor{red}{\node[vertex] (f) at  (0,-3) {$v_5$};}
\textcolor{green}{\node[vertex] (g) at  (8,3) {$v_7$};}
\textcolor{green}{\node[vertex] (h) at  (8,-3) {$v_6$};}
%\node[vertex] (a1) at (1.5,0) {};
%\node[vertex] (a2) at (3,0) {};
%edges
\textcolor{black}{\draw[edge] (c) to node[above] {\hspace{0.2cm}\begin{turn}{20}$c_{1}$\end{turn}} (f);}
\textcolor{black}{\draw[edge] (a) to node[above] {\hspace{0.2cm}\begin{turn}{20}$c_{2}$\end{turn}} (h);}
\textcolor{black}{\draw[edge] (f) to node[left] {\hspace{0.2cm}\begin{turn}{20}$c_{4}$\end{turn}} (e);}
\textcolor{black}{\draw[edge] (h) to node[right] {\hspace{0.2cm}\begin{turn}{20}$c_{5}$\end{turn}} (g);}
\textcolor{black}{\draw[edge] (g) to node[above] {\hspace{0.2cm}\begin{turn}{20}$c_{6}$\end{turn}} (a);}
\textcolor{black}{\draw[edge] (e) to node[above] {\hspace{0.2cm}\begin{turn}{20}$c_{7}$\end{turn}} (c);}
\textcolor{black}{\draw[loop] (b) to node[above] {\hspace{0.2cm}\begin{turn}{0}$c_{8}$\end{turn}} (b);}
\textcolor{black}{\draw[loop] (d) to node[below] {\hspace{0.2cm}\begin{turn}{0}$c_{9}$\end{turn}} (d);}

\end{tikzpicture}
\captionof{figure}{\textbf{4-partite directed graph $\prescript{c}{}{ \mathsf{Sq}}_{3}^{d}$}}\label{Fig::4}

\vspace{1cm}

and concerning the 
$C^*(\prescript{c}{}{\mathsf{Sq}}_{3}^{d})$ we have the following result.
\begin{prop}
    Let $\prescript{c}{}{ \mathsf{Sq}}_{3}^{d}$ be as in figure \ref{Fig::4}. Then for $\mathcal{H}$ an Hilbert space, the set $\{\mathscr{S},\mathscr{P}\}\in B(\mathcal{H})$ consisting of the following partial isometries

    \begin{align}
    &S_{c_1}=S_{c_2}=E_{3,1}, \qquad S_{c_4}=S_{c_5}=E_{1,2},\label{Par:cSq3d}\\& S_{c_7}=S_{c_6}=E_{2,3},\qquad S_{c_8}=S_{c_9}=E_{1,3}.
\end{align}
will provide a CK-$\prescript{c}{}{ \mathsf{Sq}}_{3}^{d}$ family as generators of the finite dimensional $C^*$-graph algebra $C^*(\prescript{c}{}{ \mathsf{Sq}}_{3}^{d})=C^*(\mathscr{S},\mathscr{P})$, which could be identified as $\mathcal{M}_{3,3}( C(\mathbb T^2))$ for $\mathcal{M}_{3,3}$ a subspace consisting of block matrices of the space of $6\times 6$ matrices.
    
\end{prop}
\begin{proof}
    As usual, for $i\in\{1,\cdots,8\}$, the orthogonal projections $P_{v_i}$ might have nonzero or an infinite dimension. So let us find out that which space do we live in.

    Let $\mathcal{H}$ to be the underlying Hilbert space and for $\{\mathscr{S},\mathscr{P}\}\in B(\mathcal{H})$, let us start by looking at the dimension of the subspaces $P_{v_i}\mathcal{H}$ for $i\in\{1,\cdots,8\}$, and if any of $P_{v_i}\mathcal{H}$ are infinite-dimensional, then they all have to be infinite-dimensional.

    \begin{align}
&\dim(P_{v_3}\mathcal{H})=\dim(S_{c_{7}}\mathcal{H})=\dim(P_{v_{8}}\mathcal{H}),\notag \\& \dim(P_{v_8}\mathcal{H})=\dim(S_{c_{4}}\mathcal{H})=\dim(P_{v_{5}}\mathcal{H}),\label{Eq:cSq3d:1}  \\& \dim(P_{v_5}\mathcal{H})=\dim(S_{c_{1}}\mathcal{H})=\dim(P_{v_{3}}\mathcal{H}),\notag \\& \dim(P_{v_2}\mathcal{H})=\dim(S_{c_{8}}\mathcal{H}),\notag \\& \dim(P_{v_1}\mathcal{H})=\dim(S_{c_{9}}\mathcal{H}),\label{Eq:cSq3d:2} \\& \dim(P_{v_4}\mathcal{H})=\dim(S_{c_{6}}\mathcal{H})=\dim(P_{v_{7}}\mathcal{H}),\notag \\& \dim(P_{v_7}\mathcal{H})=\dim(S_{c_{5}}\mathcal{H})=\dim(P_{v_{6}}\mathcal{H}),\label{Eq:cSq3d:3} \\& \dim(P_{v_6}\mathcal{H})=\dim(S_{c_{2}}\mathcal{H})=\dim(P_{v_{4}}\mathcal{H}).\notag
\end{align}
Note that relations (\ref{Eq:cSq3d:1}) and (\ref{Eq:cSq3d:3}) will provide us with the graph $C^*$-algebra $M_3(C(\mathbb T))$ and the relations (\ref{Eq:cSq3d:2}) will provide us with the graph $C^*$-algebra $C(\mathbb T)$. And now by considering $\mathcal{M}_{3,3}$ be a subspace of the space of $6\times 6$ matrices, consisting of block matrices $\begin{bmatrix}
    A& 0\\ 0& B
\end{bmatrix}$, for $A$ and $B$ some $3\times 3$ matrices, then the claim might be formulated as follows 

\begin{align}
    C^*(\prescript{c}{}{\mathsf{Sq}}_{3}^{d}):=C^*(\mathscr{S},\mathscr{P})&=M_3( C(\mathbb T))\oplus M_3( C(\mathbb T)) \oplus C(\mathbb T) \oplus C(\mathbb T)\notag \\& \cong
    \mathcal{M}_{3,3}( C(\mathbb T))\oplus C^2(\mathbb T)\notag \\&\cong \mathcal{M}_{3,3}( C(\mathbb T))\oplus C(\mathbb T)\notag \\& \cong \mathcal{M}_{3,3}( C(\mathbb T^2)),
\end{align}

and the family $\mathscr{S}$ would be as follows.

\begin{align*}
    &S_{c_1}=S_{c_2}=E_{3,1}, \qquad S_{c_4}=S_{c_5}=E_{1,2},\\& S_{c_7}=S_{c_6}=E_{2,3},\qquad S_{c_8}=S_{c_9}=E_{1,3}.
\end{align*}
which hopefully will provide us with the desired result!

\end{proof}
Note that this paper still has not been concluded and the computation for the higher dimensions still remains open and the extended and the revised versions will be uploaded in arXiv!

\section{Concluding remarks}
The work of this paper is a continuation of our previous works \cite{RH24,R242,R243} on studying Cuntz-Krieger graph family related problems. We started this paper by looking up for some set of regulated directed graphs (with the same number of the source and sink vertices) having a regulated odd number of Hamiltonian paths. Then we turned our attention to study set of regulated colored directed graphs, and this was without any intention, and only caused by the way of implementing our ideas.

Graph $C^*$-algebras have shown their great importance in modern mathematics, as we have seen in \cite{R242} that these very interesting examples lead us in proposing the first initial examples of what so called the multiplier Hopf $*$-graph algebras (note that the naming might not be correct, but this is the way how we call them!). The proposed examples and the later on study on the matrix representation theory of $C(S_{n}^{+})$ that will appear in our next work (still under preparation), might lead the community in answering the question raised by Wang \cite{SW99}, asking ``if there are $q$-deformation of the finite groups of Lie type in the sense of quantum groups, and how it might be proceeded!''. 

We have to admit that finding and studying such a set of regulated directed locally finite graphs might not be an easy task to do, but this is how science works, and it is the only way!

The work of this paper still has not been concluded and one might look for such set of graphs that could provide us with some interesting finite dimensional graph $C^*$-algebras!

Apart from the above missleading directions, one might also take a part in looking for the associated non-commutative (confusability) graph generated by $S_{i}^{*}S_i$, for $S_i$s the partial isometries as the generatores of the $C^*$-graph algebra $C^*(\mathscr{S},\mathscr{P})$ and the other related structures!
\section{Acknowledgement}
\hspace*{0.2cm}  
%F.R. is supported by the Azarbaijan Shahid Madani University under the grant contract No. 117.d.22844 - 08.07.2023. 
The author is partially supported by a grant from the Institute for Research in Fundamental Sciences (\textit{IPM}), with grant No. 1404140052.

%Apart from the acknowledgments related with the financial supports, F.R. also would like to express his sincere thanks and gratitude for the hospitality of the Department of Mathematics of the Azarbaijan Shahid Madani University, where all of the work conducted in this paper has been extracted and has been finalized!

%\hspace*{0.2cm} The second author was supported by the Department of Mathematics of the Azarbaijan Shahid Madani University under grant No. 1402/270 - 19.04.2023.

%\hspace*{0.2cm}The authors of this manuscript have equal contributions to this work.

\def\cprime{$'$} \def\cprime{$'$} \def\cprime{$'$}
\providecommand{\bysame}{\leavevmode\hbox to3em{\hrulefill}\thinspace}
\providecommand{\MR}{\relax\ifhmode\unskip\space\fi MR }
 %\MRhref is called by the amsart/book/proc definition of \MR.
\providecommand{\MRhref}[2]{%
	\href{http://www.ams.org/mathscinet-getitem?mr=#1}{#2}
}
\providecommand{\href}[2]{#2}

%\begingroup
\let\itshape\upshape

%\bibliographystyle{plain}
%\bibliography{\jobname}

\begin{thebibliography}{9}
\bibitem{BCEHPSW19}
Michael Brannan, Alexandru Chirvasitu, Kari Eifler, Samuel Harris, Vern Paulsen, Xiaoyu Su, and Mateusz Wasilewski.  Bigalois extensions and the graph isomorphism game. \em{Communications in Mathematical Physics,} \emph{375(3):1777–1809,} {\bf 2019}.


\bibitem{Ban20}	
 Banica, T. Quantum permutation groups. \em{arXiv} {\bf 2020}, arXiv:2012.10975.

\bibitem{BEVW22}
 Brannan, Michael, Kari Eifler, Christian Voigt, and Moritz Weber. Quantum cuntz-krieger algebras. \em{Transactions of the American Mathematical Society}, \emph{Series B 9, no. 26}, {\bf2022}: 782--826.

\bibitem{C77}
 Cuntz, Joachim Simple $C^*$-algebras generated by isometries. \em{Comm. Math. Phys.} \emph{57, no. 2,} {\bf1977}: 173--185.

\bibitem{CK80}
 J.\ Cuntz, W.\ Krieger, A class of $C^*$-algebras and topological Markov chains, \em{Invent. Math.} \emph{56.} {\bf1980}: 251--268.

\bibitem{GMR19}
 Garcia, Stephan Ramon, Matthew Okubo Patterson, and William T. Ross. Partially isometric matrices: a brief and selective survey. \em{arXiv preprint arXiv:1903.11648} {\bf(2019)}.

 \bibitem{IR05}
Raeburn, I. Graph Algebras. {\em American Mathematical Society}, \emph{(No. 103)},  {\bf 2005}, {Providence, RI, USA}.

\bibitem{KW12}
Greg Kuperberg and Nik Weaver. A von Neumann algebra approach to quantum metrics/Quantum relations. {\em American Mathematical Society,} \emph{(No. 1010)}, 2012.

\bibitem{MRV18}
Benjamin Musto, David Reutter, and Dominic Verdon. A compositional approach to quantum functions. {\em Journal of Mathematical Physics}, \emph{59(8):081706,} 2018.

\bibitem{RH24}
Razavinia, Farrokh, and Haghighatdoost, Ghorbanali. From Quantum Automorphism of (Directed) Graphs to the Associated Multiplier Hopf Algebras. {\em Mathematics}, \textbf{2024}, \emph{12.1}: 128.

\bibitem{R242}
Razavinia, Farrokh. Into Multiplier Hopf ($*$-)graph algebras. \em{arXiv} \textbf{2024}, arXiv: 2403.09787.

\bibitem{R243}
Razavinia, Farrokh. A route to quantum computing through the theory of quantum graphs. \em{arXiv} \textbf{2024}, arXiv: 2404.13773.

 
\bibitem{RV22}
Rollier, L.; Vaes, S. Quantum automorphism groups of connected locally finite graphs and quantizations of discrete groups. \em{arXiv} \textbf{2022}, arXiv:2209.03770.

\bibitem{RVW12}
Runyao Duan, Simone Severini, and Andreas Winter. Zero-error communication via quantum channels, noncommutative graphs, and a quantum Lov\'{a}sz number. {\em IEEE Transactions on Information Theory,} {\bf 2012}, {\emph 59(2)}:1164--1174.

\bibitem{SW99}
Wang, Shuzhou. Rieffel type discrete deformation of finite quantum groups. {\em Communications in mathematical physics,} 202 {\bf 199}: 291--307.

\bibitem{W21}
Nik Weaver. Quantum graphs as quantum relations. {\em The Journal of Geometric Analysis,} {\bf 2021}: 1--23.

\bibitem{Wor87}
Stanislaw L. Woronowicz. Compact matrix pseudogroups. \em{Communications in Mathematical Physics.} {\bf 1987}, {\emph{111(4)}}:613--665.

\bibitem{AD96}
Van Daele, A. Discrete quantum groups. {\em  J. Algebra} {\bf 1996}, \textit{180}, 431--444.	

\bibitem{VD94} 
Van Daele, A. Multiplier Hopf algebras. {\em Trans. Am. Math. Soc.} {\bf 1994}, {\emph{342}}, 917--932.
\end{thebibliography}

%\endgroup
%------ Example for a paper in journal:
% \bibitem{article1}
% A.~Petrunin, Parallel transportation for Alexandrov space with curvature bounded below.
% \emph{Geom. Funct. Anal.} \textbf{8} (1998), no.~1, 123--148
% \Zbl{0903.53045} \MR{1601854}

%------ Example for a book:
% \bibitem{book1}
% W.~P. Ziemer, \emph{Weakly differentiable functions}.
% Grad. Texts in Math. 120,  Springer, New York, 1989
%\Zbl{0692.46022} \MR{1014685}

%------ Example for a paper in a book:
% \bibitem{incollection1}
% J.~S. Milne, Introduction to Shimura varieties.
% In \emph{Harmonic analysis, the trace formula, and Shimura varieties},
% pp. 265--378, Clay Math. Proc. 4,
% American Mathematical Society, Providence, RI, 2005
% \Zbl{1148.14011} \MR{2192012}

%------ Example for a preprint on arXiv:
% \bibitem{preprint1}
% D.~V. Nguyen, S.~K. Chilappagari, M.~W. Marcellin, and B.~Vasic,
% LDPC codes from latin squares free of small trapping sets.
% 2010, \arxiv{1008.4177}

%------ Example for a report:
% \bibitem{report1}
% J.~Schöberl, Commuting quasi-interpolation operators.
% Technical report isc-01-10-math, Texas A\&M University, 2001,
% \url{www.isc.tamu.edu/publications-reports/tr/0110.pdf}

%------ Example for a thesis:
% \bibitem{thesis1}
% E.~Giorgi, \emph{The geometric universe}.
% Ph.D. thesis, University of Maryland, College Park, 2002

\end{document}